\documentclass[preprint,10pt,3p]{elsarticle}
\usepackage{graphicx,amsmath,amsthm,amssymb,latexsym,amsfonts,color}
\usepackage[bookmarksnumbered, colorlinks, plainpages]{hyperref}
\usepackage{xcolor}
\usepackage{sectsty}
\usepackage{comment}
\usepackage{tabularx,booktabs}
\usepackage{algorithm2e}
\usepackage{booktabs}
\usepackage{hhline,booktabs}
\usepackage{multirow}
\usepackage{soul}
\usepackage{footnotehyper}
\usepackage{tikz}
\usepackage{natbib}
\usepackage{pifont}
\def\checkmark{\tikz\fill[scale=0.4](0,.35) -- (.25,0) -- (1,.7) -- (.25,.15) -- cycle;}
\definecolor{astral}{RGB}{46,116,181}
\sectionfont{\color{astral}}
\newcommand{\xmark}{\ding{56}}%

\newtheorem{thm}{Theorem}[section]

\newtheorem{prop}[thm]{Proposition}

\newtheorem{rem}[thm]{\bf{Remark}}
\newtheorem{example}[thm]{\bf{Example}}

\definecolor{fgreen}{rgb}{0.13, 0.55, 0.13}

   \newcommand{\fb}[1]{\textcolor{black}{#1}} 
\begin{document}
	
	\date{}

	\begin{frontmatter}\title{Accelerating iterative solvers via a two-dimensional minimum residual technique}
		
		\author[add1]{Fatemeh P. A. Beik}
		\ead{f.beik@vru.ac.ir}
		\author[add2]{Michele Benzi}
		\ead{michele.benzi@sns.it}
		\author[add1]{Mehdi Najafi-Kalyani}
		\ead{m.najafi.uk@gmail.com}
		
		\address[add1]{Department of Mathematics, Vali-e-Asr University of Rafsanjan, P.O. Box 518, Rafsanjan, Iran}
		\address[add2]{Scuola Normale Superiore, Piazza dei Cavalieri, 7, 56126, Pisa, Italy}

		\begin{abstract}
			This paper deals with speeding up the convergence of a class of two-step iterative methods 
		for solving linear systems of equations. To implement the acceleration technique, the residual norm associated with computed approximations for each sub-iterate is minimized over a certain two-dimensional subspace. Convergence properties of the proposed method are studied in detail. The  approach is further developed to solve (regularized) normal equations arising from the discretization of ill-posed problems. The results of numerical experiments are reported to illustrate the performance of exact and inexact variants of the method on several test problems from different application areas.    
		\end{abstract}
		

	\end{frontmatter}

	\section{Introduction}
	
	We consider the solution of large linear systems
	of equations of the form 
	\begin{equation}\label{main}
		Ax=b\,,
	\end{equation}
	where $A\in \mathbb{R}^{n\times n}$ is nonsingular,  for a given right-hand side vector $b\in \mathbb{R}^n$.   {We focus on the case where $A$ is 
	large and sparse (or data sparse), so that matrix-vector products with $A$ can be performed efficiently; apart from this, we make no assumptions 
	on the structure or spectral properties of $A$.}
	Under these assumptions, iterative solution methods can be a valid alternative to direct approaches, see \cite{Saad}. In particular,
	Krylov subspace methods such as the (preconditioned)  Generalized Minimal Residual (GMRES) method have been among the most effective 
	and popular iterative solvers. As is known, most Krylov methods need the computation of an orthonormal basis for the Krylov subspace, which can be costly.
	On the other hand, stationary iterative solvers do not necessitate orthogonalization, but they may converge too slowly or fail to converge entirely
	 in the absence of acceleration techniques. 
	Here,  a class of two-step iterative methods is considered. We combine this type of methods with a two-dimensional minimum residual technique, 
	with the goal of  speeding up the convergence of two-step iterative methods for solving \eqref{main}.

	In order to describe our approach in more detail, we first introduce some background notions and review some results from the literature.
	 Given a matrix $W\in \mathbb{R}^{n\times n}$, the symmetric and skew-symmetric 
	parts of $W$ are respectively defined  by
	\[\mathcal{H}(W)=\frac{1}{2}(W+W^T)\quad \text{and} \quad \mathcal{S}(W)=\frac{1}{2}(W-W^T).\]
	When the spectrum of $W$ is real,  its minimum and maximum eigenvalues  are denoted by $\lambda_{\min}(W)$ and $\lambda_{\max}(W)$, respectively. When $W$ is symmetric positive definite (SPD), we write $W\succ 0$. 
	For vectors $x, y\in \mathbb{C}^{n}$, the notation $\left\langle {x,y} \right\rangle$ refers to the Euclidean inner product of $x$ and $y$, i.e., $\left\langle {x,y} \right\rangle=x^*y$ where $x^*$ denotes the conjugate transpose
	of $x$. The Euclidean vector norm (2--norm) and its induced matrix norm are denoted by $\|\cdot \|$. The identity matrix (whose size should be clear from the
	context) will be denoted by $I$.
	In addition, we write $[x; y]$ to denote the column vector $(x^T , y^T)^T$. The field of values (FoV) of the given matrix $W\in \mathbb{R}^{n\times n}$ is given by
	\[
	\mathcal{F}(W):=\left\{  \frac{\left\langle {Wy,y} \right\rangle}{\left\langle {y,y} \right\rangle} ~\Big |~  0\ne y\in \mathbb{C}^{n}  \right\}.
	\]

	The following well-known Hermitian and skew-Hermitian splitting (HSS) method
	\begin{equation*}
		\begin{cases}
			(\alpha I + \mathcal{H}(A) ) x^{(k+\frac{1}{2})}=(\alpha I -\mathcal{S}(A) ) x^{(k)}+b\\
			(\alpha I +\mathcal{S}(A) ) x^{(k+1)} = (\alpha I - \mathcal{H}(A) ) x^{(k+\frac{1}{2})}+b
		\end{cases}\quad 	(k=0,1,2,\ldots)
	\end{equation*}
	was first proposed  in \cite{Bai}, where it is shown that if $\mathcal{H}(A)$ is positive definite, the HSS method converges to the unique solution of \eqref{main} for any initial guess and  $\alpha >0$. In \cite{Benzi2004}, the HSS method  was extended to generalized saddle point problems in which the Hermitian part of the coefficient matrix of the  system  is possibly singular.
	Recently, using a one dimensional minimum residual technique,
	Yang et al. \cite{Yang} proposed the minimum residual HSS (MRHSS) iterative method for solving \eqref{main}. The MRHSS method constructs the sequence of approximate solutions
	$\{x^{(k)}\}_{k=0}^{\infty}$ by the following two-step iterative method:
	\begin{subequations}\label{bit}
		\begin{align}
			x^{(k+\frac{1}{2})} & = x^{(k)} +\beta_k(\alpha I +\mathcal{H}(A))^{-1}r^{(k)}\label{bita}\\
			x^{(k+1)} & =  x^{(k+\frac{1}{2})} +\gamma_k(\alpha I +\mathcal{S}(A))^{-1}r^{(k+\frac{1}{2})}\label{bitb}
		\end{align}
	\end{subequations}
	where
	\begin{equation}\label{eq261}
			{\beta _k} = \frac{{\left\langle {{r^{(k)}},A{\delta ^{(k)}}} \right\rangle }}{{{{\| {A{\delta ^{(k)}}} \|}^2}}}    
		\quad { \text{and}} \quad {{\gamma _k} = \frac{{\left\langle {{r^{(k + \frac{1}{2})}},A{\delta ^{(k + \frac{1}{2})}}} \right\rangle }}{{{{\| {A{\delta ^{(k + \frac{1}{2})}}} \|}^2}}}},
	\end{equation}
	in which $\delta ^{(k)}=(\alpha I +\mathcal{H}(A))^{-1}r^{(k)}$ and $\delta ^{(k + \frac{1}{2})}=(\alpha I +\mathcal{S}(A))^{-1}r^{(k+\frac{1}{2})}$, 
	with $r^{(k)} = b - Ax^{(k)}$ and $r^{(k+\frac{1}{2})}= b - Ax^{(k+\frac{1}{2})}$. 
	These values of the parameters $\beta_k$ and $\gamma_k$ are obtained by minimizing the associated residual norms over a certain subspace at each step. The reported numerical results illustrate the effectiveness of the MHRSS method in comparison to some of the existing approaches in the literature; see \cite[Section 4]{Yang}. The convergence of this method is ensured under the following necessary and sufficient condition:
	\begin{equation}\label{eqn1}
		0 \notin \mathcal{F}(A(\alpha I +\mathcal{H}(A))^{-1}) ~\cap~  \mathcal{F}(A(\alpha I +\mathcal{S}(A))^{-1}).
	\end{equation}
	In general, it is not easy to check  the above condition. Therefore, Yang \cite{Yang2} shows that if the second parameter is determined by minimizing an alternative norm, then the resulting iterative scheme is unconditionally convergent. More precisely, instead of the second formula in \eqref{eq261}, the parameter $\gamma_k$ is computed as follows:
	\begin{equation}
		{{\gamma _k} = \frac{{\left\langle {{Mr^{(k + \frac{1}{2})}},MA{\delta ^{(k + \frac{1}{2})}}} \right\rangle }}{{{{\| {MA{\delta ^{(k + \frac{1}{2})}}} \|}^2}}}},
	\end{equation}
	which is the minimizer of
	\[\mathop {\min }\limits_\gamma   \| {{r^{(k + \frac{1}{2})}} - \gamma A{\delta ^{(k + \frac{1}{2})}}}  \|_M.\]
	Here $M=(\alpha I + \mathcal{H}(A))^{-1}$ and $\|x\|_M:=\|Mx\|$. Although this variant of the method  is competitive with the MRHSS method in term of required number of iterations to achieve a given residual tolerance, it consumes more CPU time due to the higher computational costs resulting from the weighted inner product.
	
	The following proposition is a direct consequence of a result established in \cite[Proposition  2.4]{Golub}.
	It shows that, under a certain assumption on the extreme eigenvalues of $\mathcal{H}(A)$, we can find a shift $\eta$ for which the condition \eqref{eqn1} is satisfied.

	\begin{prop}\label{prop2.5}
		Let $A\in \mathbb{R}^{n\times n}$, and let ${\lambda _{\max }}$ and ${\lambda _{\min }}$ be the largest
	and smallest eigenvalues of $\mathcal{H}(A)$. If
	${\lambda _{\max }}{\lambda _{\min }} >- {\lambda _{\min }} (\mathcal{S}^T(A)\mathcal{S}(A)),$
	then there exists an $\eta$ for which $\|\tilde{\mathcal{S}}^{-1}(A)\|_2 \|\tilde{\mathcal{H}}^{}(A)\|_2 <1$
	where $\tilde{\mathcal{H}}(A)= \mathcal{H}(A)-\eta I$ and $\tilde{\mathcal{S}}(A)=\eta I + \mathcal{S}(A)$. In particular, the parameter $\eta$ can be chosen by
		\begin{equation}\label{eta}
			\eta^*=\frac{{\lambda _{\max }}+{\lambda _{\min }}}{2},
		\end{equation}
		for which the value of ${\|\tilde{\mathcal{S}}^{-1}(A)\|}_2 {\|\tilde{\mathcal{H}}^{}(A)\|}_2$ is minimized.
	\end{prop}
	The previous proposition
	guarantees the existence of a shift $\eta$ such that
	$0 \notin \mathcal{F}(A(\eta I +\mathcal{S}(A))^{-1})$. We comment that the assumption in Proposition \ref{prop2.5} holds when $\mathcal{H}(A)$ is positive definite.
	
	Motivated by the conclusion of the preceding proposition, the following iterative method was proposed in \cite{Ameri},
	\begin{subequations}\label{methodold}
		\begin{align}
			x^{(k+\frac{1}{2})} & =  x^{(k)} +\beta_k(\alpha I +\mathcal{H}(A))^{-1}r^{(k)}\label{methodolda}\\
			x^{(k+1)} & =  x^{(k+\frac{1}{2})} +\gamma_k(\eta I +\mathcal{S}(A))^{-1}r^{(k+\frac{1}{2})}, \label{methodoldb}
		\end{align}
	\end{subequations}
	where $\beta_k$ and $\gamma_k$ are given by \eqref{eq261} and $r^{(k)}$, $r^{(k+\frac{1}{2})}$ are the residuals at steps $k$ and $k+\frac12$.
	The parameter $\alpha$ is again chosen to be positive,
	and the parameter $\eta$ is assigned the value 
	$\eta^*$ given by \eqref{eta}.
	It was shown in \cite{Ameri} that choosing the parameters $\alpha$ and $\eta$ in this way results in an unconditionally convergent method, which was 
	experimentally observed to be competitive with the MRHSS method. 	We comment that when 
	$\mathcal{H}(A) \succ 0$, the iterative method \eqref{methodold} remains convergent if we omit the parameter $\alpha$ in \eqref{methodolda},
	i.e., for $\alpha =0$. Indeed, it is not difficult to verify that the iterative method 
	\begin{subequations}\label{methodoldr}
		\begin{align}
			x^{(k+\frac{1}{2})} & =  x^{(k)} +\beta_k \mathcal{H}(A)^{-1}r^{(k)}\label{methodoldra}\\
			x^{(k+1)} & =  x^{(k+\frac{1}{2})} +\gamma_k(\eta^* I +\mathcal{S}(A))^{-1}r^{(k+\frac{1}{2})} \label{methodoldrb},
		\end{align}
	\end{subequations}
where the parameters $\beta_k$ and $\gamma_k$ are obtained by \eqref{eq261}, is convergent under the following condition:
\[
0 \notin \mathcal{F}(A\mathcal{H}(A))^{-1})~\cap~  \mathcal{F}(A(\eta^* I +\mathcal{S}(A))^{-1})\,.
\]
	
	In this paper, we develop a new class of two-step iterative methods for solving \eqref{main}. The proposed approach depends on two given splittings of the coefficient matrix and  benefits from a two-dimensional minimum residual technique at each sub-step. First, the convergence properties of the proposed method are analyzed in detail for solving 
	nonsingular  linear systems of equations in general form. Then, the method is used to solve a certain class of  augmented two-by-two block systems of equations  (denoted by $Kx=b$)  corresponding to regularized discrete linear ill-posed problems (with Tikhonov regularization). We also conduct numerical experiments aimed at assessing the performance of the proposed algorithm as an iterative regularization method.  The augmented system formulation provides an approximation to the (least-squares) solution of $Af=g$ appearing in the discretization of ill-posed problems in which $A$ is possibly non-square. As mentioned before, the proposed method relies on two splittings of the coefficient matrix $K$ of
	the augmented system. It turns out that the FoV of $K\hat{M}^{-1}$ plays a key role in determining the convergence rate of the proposed method for our specific choices of splittings $K=\tilde{M}-\tilde{N}=\hat{M}-\hat{N}$. Therefore, some bounds for the FoV of $K\hat{M}^{-1}$ are obtained theoretically and verified experimentally.  
	Numerical experiments are reported to compare the performance of the proposed approach with some methods found in the literature. 
	We emphasize that in the implementation of the new approach one does not have to deal with the difficulty of determining suitable values of relaxation parameters, unlike  in some of the existing methods \cite{Cui,Yang}.

	The remainder of the paper is organized as follows: In Sect.~\ref{sec2} we propose a new method, {called Two-Step Two-Dimensional Minimum Residual
	(TSTMR) method.} Under a sufficient condition, we show that the breakdown of the method is a ``lucky" breakdown\footnote{A breakdown in an iterative method is called a lucky breakdown if we are able to find the exact solution using the approximate solutions obtained in previous steps of the method once the breakdown occurs.} and the convergence of the method is proved in case of no breakdown.
	The convergence of the method is further analyzed in Sect.~\ref{sec3} where the proposed approach is implemented for an augmented system arising from the discretization of ill-posed problems. In Sect.~\ref{sec4}, we report some numerical results to compare the performance of TSTMR method with some of the recently proposed iterative methods on several test problems.
	Finally, some brief conclusive remarks are given in Sect.~\ref{sec5}.
	
	\section{Proposed method and its convergence analysis}\label{sec2}
	
	In this section we establish a class of two-step iterative methods  for solving \eqref{main} where each sub-step in the main iteration involves a two-dimensional minimum residual search. To construct such a method, two prescribed splittings of $A$ are involved, $A=\tilde{M}-\tilde{N}=\hat{M}-\hat{N}$. The performance of the proposed method relies on the choice of splittings and on a two-dimensional subspace over which the norm of residuals is minimized. The method, referred to as TSTMR in the following, produces
a sequence of approximate solutions $\{x^{(k)}\}_{k=1}^{\infty}$ as follows:
	\begin{subequations}\label{method}
		\begin{align}
			{x^{(k + \frac{1}{2})}} &= {x^{(k)}} + {\beta^{(k)}_1}{\delta ^{(k)}_1}+ {\beta^{(k)}_2}{\delta ^{(k)}_2}\label{methoda}\\
			{x^{(k + 1)}} &= {x^{(k + \frac{1}{2})}} + {\gamma^{(k)}_1}{\delta_1 ^{(k + \frac{1}{2})}}+ {\gamma^{(k)}_2}{\delta_2 ^{(k + \frac{1}{2})}}\label{methodb}
		\end{align}
	\end{subequations}
	in which
	\begin{equation}\label{be1}
		\delta ^{(k)}_1 = \tilde{M}^{-1} r^{(k)},\quad  \delta ^{(k)}_2 = \delta ^{(k)}_1-\delta ^{(k-1)}_1,\quad  \delta_1 ^{(k + \frac{1}{2})} = \hat{M}^{-1} r^{(k+ \frac{1}{2})}\quad \text{and} \quad \delta_2 ^{(k + \frac{1}{2})} =\delta_1 ^{(k + \frac{1}{2})} -\delta_1 ^{((k-1) + \frac{1}{2})}
	\end{equation}
	with $r^{(k)} = b - Ax^{(k)}$ and $r^{(k+\frac{1}{2})}= b - Ax^{(k+\frac{1}{2})}$. 
	The parameters $\beta_i^{(k)}$ and $\gamma^{(k)}_i$ (for $i=1,2$)
	are the solutions of certain two-by-two linear systems of equations which are specified later in this section.
	The approximate solution $x^{(1)}$ is determined by using the following two steps:
	\begin{subequations}\label{firststep}
		\begin{align}
			x^{(\frac{1}{2})} & = x^{(0)} +\beta_0\delta ^{(0)} \label{firststepa}\\
			x^{(1)} & =  x^{(\frac{1}{2})} +\gamma_0\delta ^{(\frac{1}{2})} \label{firststepb}
		\end{align}
	\end{subequations}
	where
	\begin{equation*}
\beta _0 = \frac{{\left\langle {{r^{(0)}},A{\delta ^{(0)}}} \right\rangle }}{{{{\| {A{\delta ^{(0)}}} \|}^2}}}    	
		\quad { \text{and}} \quad {{\gamma _0} = \frac{{\left\langle {{r^{(\frac{1}{2})}},A{\delta ^{( \frac{1}{2})}}} \right\rangle }}{{{{\| {A{\delta ^{(\frac{1}{2})}}} \|}^2}}}} \,.
	\end{equation*}
	Here $\delta ^{(0)} = \tilde{M}^{-1} r^{(0)}$, $r^{(0)}=b-Ax^{(0)}$, $\delta^{(\frac{1}{2})} = \hat{M}^{-1} r^{(\frac{1}{2})}$, $r^{(\frac{1}{2})}=b-Ax^{(\frac{1}{2})}$, and the arbitrary  initial guess $x^{(0)}$ is given.
	If 	$0 \notin \mathcal{F}(A\tilde{M}^{-1}) ~\cap~  \mathcal{F}(A\hat{M}^{-1})$,  \fb{by the following remark, it turns out} that either $
	\|r^{(1)}\|\le \|r^{(\frac{1}{2})}\| < \|r^{(0)}\|
	$
	or
	$
	\|r^{(1)}\| <  \|r^{(\frac{1}{2})}\| \le \|r^{(0)}\|
	$
	is valid.
	
	\begin{rem}
\fb{\rm	Let $\|r^{(0)}\|\ne 0$.	Considering the values of $	\beta _0$ and $\gamma _0$, it can be seen that
		$$\left\langle {{r^{(\frac{1}{2})}},A{\delta ^{(0)}}} \right\rangle  = 0\quad \text{and} \quad \left\langle {{r^{(1)}},A{\delta ^{(\frac{1}{2})}}} \right\rangle  = 0.$$
Evidently, we have $r^{(\frac{1}{2})}=r^{(0)}-\beta_0A\delta^{(0)}$ and $r^{(1)}=r^{(\frac{1}{2})}-\gamma_0A\delta^{(\frac{1}{2})}$. Hence,	using some algebraic computations, it can be seen that
	\begin{eqnarray*}
	\left\langle {{r^{(\frac{1}{2})}},{r^{(\frac{1}{2})}}} \right\rangle  &= &\left\langle {{r^{(0)}} - {\beta _0}A{\delta ^{(0)}},{r^{(\frac{1}{2})}}} \right\rangle	\\
	& = & \left\langle {{r^{(0)}},{r^{(\frac{1}{2})}}} \right\rangle\\
	& = & \left\langle {{r^{(0)}},{r^{(0)}} - {\beta _0}A{\delta ^{(0)}}} \right\rangle \\
	& = & \left\langle {{r^{(0)}},{r^{(0)}}} \right\rangle  - \beta _0\left\langle {{r^{(0)}},{}A{\delta ^{(0)}}} \right\rangle \\
	& = & \left\langle {{r^{(0)}},{r^{(0)}}} \right\rangle  - \frac{{\left\langle {{r^{(0)}},A{\delta ^{(0)}}} \right\rangle }}{{{{\left\| {A{\delta ^{(0)}}} \right\|}^2}}}\left\langle {{r^{(0)}},A{\delta ^{(0)}}} \right\rangle.
\end{eqnarray*}
		Given that $\delta ^{(0)} = \tilde{M}^{-1} r^{(0)}$, we get 
		\begin{eqnarray}
			\nonumber {\left\| {{r^{(\frac{1}{2})}}} \right\|^2} &= &{\left\| {{r^{(0)}}} \right\|^2} - \frac{{\left\langle {{r^{(0)}},A{{\tilde M}^{ - 1}}{r^{(0)}}} \right\rangle }}{{{{\left\| {A{{\tilde M}^{ - 1}}{r^{(0)}}} \right\|}^2}}}\left\langle {{r^{(0)}},A{{\tilde M}^{ - 1}}{r^{(0)}}} \right\rangle \\
			\nonumber & = &{\left\| {{r^{(0)}}} \right\|^2}\left( {1 - \frac{{{{\left\langle {{r^{(0)}},A{{\tilde M}^{ - 1}}{r^{(0)}}} \right\rangle }^2}}}{{{{\left\| {{r^{(0)}}} \right\|}^4}}}\frac{{{{\left\| {{r^{(0)}}} \right\|}^2}}}{{{{\left\| {A{{\tilde M}^{ - 1}}{r^{(0)}}} \right\|}^2}}}} \right)\\
			& = & {\left\| {{r^{(0)}}} \right\|^2}\left( {1 - \frac{{{{\left\langle {{r^{(0)}},A{{\tilde M}^{ - 1}}{r^{(0)}}} \right\rangle }^2}}}{{{{\left\langle {{r^{(0)}},{r^{(0)}}} \right\rangle }^2}}}\frac{{{{\left\| {{r^{(0)}}} \right\|}^2}}}{{{{\left\| {A{{\tilde M}^{ - 1}}{r^{(0)}}} \right\|}^2}}}} \right).\label{r3}
		\end{eqnarray}
	Having in mind that 
		\[\frac{{\left\langle {{r^{(0)}},A{{\tilde M}^{ - 1}}{r^{(0)}}} \right\rangle }}{{\left\langle {{r^{(0)}},{r^{(0)}}} \right\rangle }} \in\mathcal{F}(A\tilde{M}^{-1}),\]
		from Eq. \eqref{r3}, we conclude 
		\[
		{\left\| {{r^{(\frac{1}{2})}}} \right\|} \le \left\| {{r^{(0)}}} \right\|
		\]
	Notice that the above inequality holds strictly (i.e., 	$
	{\left\| {{r^{(1/2)}}} \right\|} <\left\| {{r^{(0)}}} \right\|
	$) when $0\notin \mathcal{F}(A\tilde{M}^{-1})$. We can further observe that
		\begin{equation} \label{r4}
			{\left\| {{r^{(1)}}} \right\|^2} = {\left\| {{r^{(\frac{1}{2})}}} \right\|^2}\left( {1 - \frac{{{{\left\langle {{r^{(\frac{1}{2})}},A{{\hat M}^{ - 1}}{r^{(\frac{1}{2})}}} \right\rangle }^2}}}{{{{\left\langle {{r^{(\frac{1}{2})}},{r^{(\frac{1}{2})}}} \right\rangle }^2}}}\frac{{{{\left\| {{r^{(\frac{1}{2})}}} \right\|}^2}}}{{{{\left\| {A{{\hat M}^{ - 1}}{r^{(\frac{1}{2})}}} \right\|}^2}}}} \right)\end{equation} 
		in which 
		\[\frac{{\left\langle {{r^{(0)}},A{{\hat M}^{ - 1}}{r^{(0)}}} \right\rangle }}{{\left\langle {{r^{(0)}},{r^{(0)}}} \right\rangle }}\in \mathcal{F}(A\hat{M}^{-1}). \]
		So in the case that $0\notin \mathcal{F}(A\hat{M}^{-1})$, from Eq. \eqref{r4}, we have 
		\[
		\|r^{(1)}\| < \|r^{(\frac{1}{2})}\|.
		\]
		Under the assumption $0 \notin \mathcal{F}(A\tilde{M}^{-1}) ~\cap~  \mathcal{F}(A\hat{M}^{-1})$, we conclude that either 
		$0 \notin \mathcal{F}(A\tilde{M}^{-1})$ or $0 \notin \mathcal{F}(A\hat{M}^{-1})$ is satisfied. This shows that either $
		\|r^{(1)}\|\le \|r^{(\frac{1}{2})}\| < \|r^{(0)}\|
		$
		or
		$
		\|r^{(1)}\| <  \|r^{(\frac{1}{2})}\| \le \|r^{(0)}\|
		$
		is valid.}
	\end{rem}

	In the sequel, we consider the following $2\times 2$ Gram matrix
	\[W_A(x,y): = \left[ {\begin{array}{*{20}{c}}
			{\left\langle {Ax,Ax} \right\rangle }&{\left\langle {Ay,Ax} \right\rangle }\\
			{\left\langle {Ax,Ay} \right\rangle }&{\left\langle {Ay,Ay} \right\rangle }
	\end{array}} \right].\]
	It is well-known that $W_A(x,y)$ is SPD if and only if $x$ and $y$ are linearly independent.
	
	\begin{prop}\label{prop2.1}
		Suppose that $\delta ^{(k)}_1$ and $\delta ^{(k)}_2$ at the $k$th step of \eqref{method} are nonzero vectors.
		If $W_A(\delta ^{(k)}_1,\delta ^{(k)}_2)$ is singular then there exists a scalar $\nu$ such that
		\[
		x^*=(1-\nu) x^{(k)} + \nu x^{(k-1)}
		\]
		is the exact solution of \eqref{main}.
	\end{prop}
	
	\begin{proof}
		The singularity of $W_A(\delta ^{(k)}_1,\delta ^{(k)}_2)$ implies that
		$\delta ^{(k)}_1 $ and $\delta ^{(k)}_2$ are linearly dependent. That is,
		there exists a (nonzero) scalar $\nu$ such that
		$
		\delta ^{(k)}_1 = \nu \delta ^{(k)}_2.
		$ \fb{Therefore, we have
\begin{eqnarray*}
	\tilde{M}^{-1} r^{(k)} & =& \nu \left( \delta_1^{(k)} - \delta_1^{(k-1)}\right) \\
	 & =&\nu \left( 	\tilde{M}^{-1}r^{(k)} -	\tilde{M}^{-1} r^{(k-1)}\right)
\end{eqnarray*}}
\noindent which implies that
		$$
		r^{(k)} = \nu  (r^{(k)}-r^{(k-1)} ).
		$$
		Now, we deduce that
		\begin{eqnarray*}
			b-Ax^{(k)}&=&\nu (b-Ax^{(k)}-b+Ax^{(k-1)})\\
			&=&   A (\nu x^{(k-1)}-\nu x^{(k)})\,,
		\end{eqnarray*}
		or, equivalently, that
		$
		b =  A ((1-\nu) x^{(k)}+\nu x^{(k-1)})\,,
		$
		which completes the proof.
	\end{proof}
	
	By  a reasoning similar  to that used in the proof of the above proposition, we obtain the following
	result.
	
	\begin{prop}\label{prop2.12}
		Suppose that $\delta ^{(k+\frac{1}{2})}_1$ and $\delta ^{(k+\frac{1}{2})}_2$ at the  $k$th step of \eqref{method} are nonzero vectors.
		If the matrix $W_A(\delta ^{(k+\frac{1}{2})}_1,\delta ^{(k+\frac{1}{2})}_2)$ is singular, then there exists a scalar $\nu$ such that
		\[
		x^*=(1-\nu) x^{(k+\frac{1}{2})} + \nu x^{((k-1)+\frac{1}{2})}
		\]
		is the exact solution of \eqref{main}.
	\end{prop}
	
	Assume now that the  matrices $W_A(\delta ^{(k)}_1,\delta ^{(k)}_2)$ and  $W_A(\delta ^{(k+\frac{1}{2})}_1,\delta ^{(k+\frac{1}{2})}_2)$
	are both nonsingular ($k\ge 1$). 
\fb{Evidently, the residual vectors ${r^{(k + \frac{1}{2})}}$ and 
	 $r^{(k + 1)}$, corresponding to approximation solutions obtained by \eqref{method}, are given as follows:
	 	\begin{subequations}\label{methodr}
	 	\begin{align}
	 	{r^{(k + \frac{1}{2})}} &= {r^{(k)}} - {\beta^{(k)}_1}A{\delta ^{(k)}_1}- {\beta^{(k)}_2}A{\delta ^{(k)}_2}\\
	 	{r^{(k + 1)}} &= {r^{(k + \frac{1}{2})}} - {\gamma^{(k)}_1}A{\delta_1 ^{(k + \frac{1}{2})}}- {\gamma^{(k)}_2}A{\delta_2 ^{(k + \frac{1}{2})}}.
	 \end{align}
 	 	\end{subequations}}

\noindent The parameters $\beta_i^{(k)}$ and $\gamma^{(k)}_i$ ($i=1,2$) are obtained by imposing the following orthogonality conditions:
\begin{equation}\label{orth}
	{r^{(k + \frac{1}{2})}} \bot \, A{\delta_i ^{(k)}}\begin{array}{*{20}{c}}
		{}&{ \text{and}}&{{r^{(k + 1)}}}
	\end{array} \bot \, A{\delta_i ^{(k + \frac{1}{2})}}\qquad i=1,2.
\end{equation}
\fb{By applying the orthogonality conditions \eqref{orth} on Eqs. \eqref{methodr}, we get
	\[\left[ {\begin{array}{*{20}{c}}
			{\left\langle {A\delta _1^{(k)},A\delta _1^{(k)}} \right\rangle }&{\left\langle {A\delta _2^{(k)},A\delta _1^{(k)}} \right\rangle }\\
			& \\
			{\left\langle {A\delta _1^{(k)},A\delta _2^{(k)}} \right\rangle }&{\left\langle {A\delta _2^{(k)},A\delta _2^{(k)}} \right\rangle }
	\end{array}} \right]\left[ {\begin{array}{*{20}{c}}
			{\beta _1^{(k)}}\\
			\\
			{\beta _2^{(k)}}
	\end{array}} \right] = \left[ {\begin{array}{*{20}{c}}
			{\left\langle {r_{}^{(k)},A\delta _1^{(k)}} \right\rangle }\\
			\\
			{\left\langle {r_{}^{(k)},A\delta _2^{(k)}} \right\rangle }
	\end{array}} \right]\]
and
\[\left[ {\begin{array}{*{20}{c}}
		{\left\langle {A\delta _1^{(k + \frac{1}{2})},A\delta _1^{(k + \frac{1}{2})}} \right\rangle }&{\left\langle {A\delta _2^{(k + \frac{1}{2})},A\delta _1^{(k + \frac{1}{2})}} \right\rangle }\\
		& \\
		{\left\langle {A\delta _1^{(k + \frac{1}{2})},A\delta _2^{(k + \frac{1}{2})}} \right\rangle }&{\left\langle {A\delta _2^{(k + \frac{1}{2})},A\delta _2^{(k + \frac{1}{2})}} \right\rangle }
\end{array}} \right]\left[ {\begin{array}{*{20}{c}}
		{\gamma _1^{(k)}}\\
		\\
		{\gamma _2^{(k)}}
\end{array}} \right] = \left[ {\begin{array}{*{20}{c}}
		{\left\langle {r_{}^{(k+ \frac{1}{2})},A\delta _1^{(k + \frac{1}{2})}} \right\rangle }\\
		\\
		{\left\langle {r_{}^{(k+ \frac{1}{2})},A\delta _2^{(k + \frac{1}{2})}} \right\rangle }
\end{array}} \right].\\ \]
The above linear systems can be respectively reformulated as  follows:}
	\begin{subequations}\label{findpara}
		\begin{align}
			W_A(\delta ^{(k)}_1,\delta ^{(k)}_2)	\left[ {\beta_1^{(k)};\beta_2^{(k)}} \right] & = \left[ {\left\langle {{r^{(k)}},A\delta ^{(k)}_1} \right\rangle ;\left\langle {{r^{(k)}},A\delta ^{(k)}_2} \right\rangle } \right] \label{findparaa} \\
			W_A(\delta ^{(k+\frac{1}{2})}_1,\delta ^{(k+\frac{1}{2})}_2) \left[ {\gamma_1^{(k)};\gamma_2^{(k)}} \right]& = 	\left[ {\left\langle {{r^{(k + 1/2)}},A\delta_1 ^{(k + \frac{1}{2})}} \right\rangle ;\left\langle {{r^{(k + 1/2)}},A\delta_2 ^{(k + \frac{1}{2})}} \right\rangle } \right]. \label{findparab}
		\end{align}
	\end{subequations}
	
	Under a sufficient condition, we can show that the norm of the residual vectors corresponding to the approximations produced by iterative method \eqref{method} decreases monotonically  and that the solution of $Ax=b$ is obtained in the limit.	In order to show this, we first prove the following proposition.
	
	\begin{prop}\label{prop2.3}
		Let $\delta ^{(k)}_i$ and $\delta_i ^{(k + \frac{1}{2})}$  be defined by \eqref{be1} for $i=1,2$. Then, the  inequalities 
		$$
		\|A\delta ^{(k)}_1\|^2+\|A\delta ^{(k)}_2\|^2\le \|A\tilde{M}^{-1} \|^2 \left(\|r^{(k)}\|^2+\|r^{(k)}-r^{(k-1)}\|^2\right)
		$$
		and
		$$
		\|A\delta ^{(k+ \frac{1}{2})}_1\|^2+\|A\delta ^{(k+ \frac{1}{2})}_2\|^2\le \|A\hat{M}^{-1} \|^2 \left(\|r^{(k)}\|^2+\|r^{(k+ \frac{1}{2})}-r^{((k-1)+ \frac{1}{2})}\|^2\right)
		$$ are satisfied for $k\ge 1$.
		
	\end{prop}
	
	\begin{proof}
		For ease of notation, we set  $\mathcal{V}:={(A\tilde M^{ - 1})^T}(A\tilde M^{ - 1})$. Straightforward computations reveal that\fb{
		\begin{align*}
			\|A\delta ^{(k)}_1\|^2+\|A\delta ^{(k)}_2\|^2 &= 
			\|A\tilde M^{ - 1}r^{(k)}\|^2+\|A\tilde M^{ - 1}({{r^{(k)}} - {r^{(k - 1)}}})\|^2\\
			&=(r^{(k)})^T \mathcal{V}^Tr^{(k)}+ ({{r^{(k)}} - {r^{(k - 1)}}})^T\mathcal{V}^T({{r^{(k)}} - {r^{(k - 1)}}})\\
					& = 	\left\langle {{\left[ {\begin{array}{*{20}{c}}
						{{\mathcal V}{r^{(k)}}}\\
						{{\mathcal V}({r^{(k)}} - {r^{(k - 1)}})}
				\end{array}} \right],\left[ {\begin{array}{*{20}{c}}
						{{r^{(k)}}}\\
						{{r^{(k)}} - {r^{(k - 1)}}}
				\end{array}} \right]}} \right\rangle  \\
			&= \left\langle {\left[ {\begin{array}{*{20}{c}}
						{\mathcal{V}}&0\\
						0&{\mathcal{V}}
				\end{array}} \right]\left[ {\begin{array}{*{20}{c}}
						{{r^{(k)}}}\\
						{{r^{(k)}} - {r^{(k - 1)}}}
				\end{array}} \right],\left[ {\begin{array}{*{20}{c}}
						{{r^{(k)}}}\\
						{{r^{(k)}} - {r^{(k - 1)}}}
				\end{array}} \right]} \right\rangle\\
			&\le \lambda_{\max } (\mathcal{V}) \left\langle {\left[ {\begin{array}{*{20}{c}}
						{{r^{(k)}}}\\
						{{r^{(k)}} - {r^{(k - 1)}}}
				\end{array}} \right],\left[ {\begin{array}{*{20}{c}}
						{{r^{(k)}}}\\
						{{r^{(k)}} - {r^{(k - 1)}}}
				\end{array}} \right]} \right\rangle
		\end{align*}}
	
	\noindent	from which the first inequality follows.
		The validity of second relation can also be checked in a similar manner.
	\end{proof}
	\begin{thm}\label{thm2.3}
		Let $A$, $\tilde M$ and $\hat M$ be $n\times n$ real matrices such that
		\begin{equation}\label{eq11}
			0 \notin \mathcal{F}(A\tilde{M}^{-1}) ~\cap~  \mathcal{F}(A\hat{M}^{-1}).
		\end{equation}
		Assume that $x^{(\frac{1}{2})}$ and $x^{(1)}$ are computed by \eqref{firststepa} and \eqref{firststepb}, respectively, for an arbitrary given initial guess $x^{(0)}.$
		If $W_A(\delta ^{(k)}_1,\delta ^{(k)}_2)$ and $W_A(\delta ^{(k+\frac{1}{2})}_1,\delta ^{(k+\frac{1}{2})}_2)$ are nonsingular for $k=1,2,\ldots,$ then there exists a positive constant $\mathcal{L}_k < 1$ such that
		$$
		\|r^{(k+1)}\| \le \mathcal{L}_k \|r^{(k)}\| \quad \text{for} \quad k\ge 0.
		$$
	\end{thm}
	
	\begin{proof}
		Let $\lambda_M=\lambda_{\max}\left(W_A(\delta ^{(k)}_1,\delta ^{(k)}_2)\right)$.  It is not difficult to observe that
		\[
		\lambda_M \le \|A\delta ^{(k)}_1\|^2+\|A\delta ^{(k)}_2\|^2.
		\]
		The above inequality together with Eqs. \eqref{methoda} and \eqref{findparaa} imply that
		\begin{align*}
			\nonumber	\|r^{(k+\frac{1}{2})}\|^2&=\|r^{(k)}\|^2-\left[ {\left\langle {{r^{(k)}},A\delta ^{(k)}_1} \right\rangle ;\left\langle {{r^{(k)}},A\delta ^{(k)}_2} \right\rangle } \right]^T 	\left(W(\delta ^{(k)}_1,\delta ^{(k)}_2)\right)^{-1} \left[ {\left\langle {{r^{(k)}},A\delta ^{(k)}_1} \right\rangle ;\left\langle {{r^{(k)}},A\delta ^{(k)}_2} \right\rangle } \right]	\\
			\nonumber & \le \|r^{(k)}\|^2 -\frac{1}{\lambda_M }\left[ {\left\langle {{r^{(k)}},A\delta ^{(k)}_1} \right\rangle ;\left\langle {{r^{(k)}},A\delta ^{(k)}_2} \right\rangle } \right]^T  \left[ {\left\langle {{r^{(k)}},A\delta ^{(k)}_1} \right\rangle ;\left\langle {{r^{(k)}},A\delta ^{(k)}_2} \right\rangle } \right]\\
			\nonumber & \le \|r^{(k)}\|^2 -\frac{1}{\|A\delta ^{(k)}_1\|^2+\|A\delta ^{(k)}_2\|^2 }\left( {\left\langle {{r^{(k)}},A\delta ^{(k)}_1} \right\rangle^2 +\left\langle {{r^{(k)}},A\delta ^{(k)}_2} \right\rangle^2 } \right).
		\end{align*}
		By Proposition \ref{prop2.3}, we have
		\begin{align}
			\nonumber	\|r^{(k+\frac{1}{2})}\|^2 & \le \|r^{(k)}\|^2 -\frac{1}{\|A\tilde{M}^{-1} \|^2 \left(\|r^{(k)}\|^2+\|r^{(k)}-r^{(k-1)}\|^2\right) }\left( {\left\langle {{r^{(k)}},A\delta ^{(k)}_1} \right\rangle^2 +\left\langle {{r^{(k)}},A\delta ^{(k)}_2} \right\rangle^2 } \right).  \\
			& \le \|r^{(k)}\|^2  \left(1- \frac{1}{\|A\tilde{M}^{-1} \|^2 }\cdot \frac{\tilde{\xi}^2\|r^{(k)}\|^2}{\|r^{(k)}\|^2+\|r^{(k)}-r^{(k-1)}\|^2 }\right) \label{eq14}
		\end{align}
		where
	\begin{equation}\label{minfov}
			\tilde{\xi}: = \min 	\left\{ \frac{\left|{{\left\langle {A\tilde{M}^{-1} y,y} \right\rangle}}\right|}{\left\langle {y,y} \right\rangle}   ~\text{for}~  0\ne y\in \mathbb{R}^{n}  \right\},
	\end{equation}
		having in mind that $\left\langle {{r^{(k)}},A\delta ^{(k)}_1} \right\rangle=\left\langle {{r^{(k)}},A\tilde{M}^{-1}{{r^{(k)}}}} \right\rangle$.
		Exploiting a similar strategy, we can deduce that
		\begin{align}
			\|r^{(k+1)}\|^2
			& \le \|r^{(k+\frac{1}{2})}\|^2  \left(1- \frac{1}{\|A\hat{M}^{-1} \|^2 }\cdot \frac{\hat{\xi}^2\|r^{(k+\frac{1}{2})}\|^2 }{\|r^{(k+\frac{1}{2})}\|^2+\|r^{(k+\frac{1}{2})}-r^{((k-1)+\frac{1}{2})}\|^2 }\right)\label{eq15}
		\end{align}
		with
		\[
		\hat{\xi} := \min 	\left\{ \frac{\left|{{\left\langle {A\hat{M}^{-1} y,y} \right\rangle}}\right|}{\left\langle {y,y} \right\rangle}   ~\text{for}~  0\ne y\in \mathbb{R}^{n}  \right\}.
		\]
		For ease of notation, we set
		\begin{equation}\label{Lth}
		\tilde L_{k}:=\frac{\tilde{\xi}^2\|r^{(k)}\|^2}{\|r^{(k)}\|^2+\|r^{(k)}-r^{(k-1)}\|^2 }, \quad \hat L_{k}:=\frac{\hat{\xi}^2\|r^{(k+\frac{1}{2})}\|^2 }{\|r^{(k+\frac{1}{2})}\|^2+\|r^{(k+\frac{1}{2})}-r^{((k-1)+\frac{1}{2})}\|^2 }\,.
		\end{equation}
		By definition of ${\tilde L}_k$, we have
	\begin{eqnarray}
		\frac{{{{\tilde L}_k}}}{{\| {A\tilde M^{ - 1}} \|^2}}&= &\frac{1}{{\| {A\tilde M^{ - 1}} \|^2}} \cdot \frac{\tilde{\xi}^2\|r^{(k)}\|^2}{\|r^{(k)}\|^2+\|r^{(k)}-r^{(k-1)}\|^2 }\,.\label{eq20n}
	\end{eqnarray}
	Using the definition of \fb{ $\tilde{\xi}^2$ in \eqref{minfov}, we get
	\[
	\tilde{\xi}^2 \le  \frac{{{\left\langle {A\tilde{M}^{-1} r^{(k)},r^{(k)}} \right\rangle}}^2}{\left\langle {r^{(k)},r^{(k)}}\right\rangle^2 }
	\]
which implies that 
\[
	\tilde{\xi}^2 \|r^{(k)}\|^2\le  \frac{{{\left\langle {A\tilde{M}^{-1} r^{(k)},r^{(k)}} \right\rangle}}^2}{\|r^{(k)}\|^2}.
\]
}
Recalling that $\delta ^{(k)}_1=\tilde M^{ - 1}r^{(k)}$, \fb{from Eq. \eqref{eq20n}, we can conclude that}
	\begin{eqnarray*}
		\frac{{{{\tilde L}_k}}}{{\| {A\tilde M^{ - 1}} \|^2}}&\le & \frac{1}{{\| r^{(k)}\| ^2}} \cdot \frac{\left\langle {{r^{(k)}},A\delta ^{(k)}_1} \right\rangle^2  }{{\| {A\tilde M^{ - 1}} \|^2}\left(\|r^{(k)}\|^2+\|r^{(k)}-r^{(k-1)}\|^2 \right)}.
	\end{eqnarray*}
	Now Proposition \ref{prop2.3} ensures that
	\begin{eqnarray*}
		\frac{{{{\tilde L}_k}}}{{\| {A\tilde M^{ - 1}} \|^2}}&\le & \frac{1}{{\| r^{(k)}\| ^2}} \cdot \frac{\left\langle {{r^{(k)}},A\delta ^{(k)}_1} \right\rangle^2  }{	\|A\delta ^{(k)}_1\|^2+\|A\delta ^{(k)}_2\|^2}\\
		&\le & \frac{1}{{\left\| r^{(k)}\right\| ^2}} \cdot \frac{\left\langle {{r^{(k)}},A\delta ^{(k)}_1} \right\rangle^2  }{	\|A\delta ^{(k)}_1\|^2}.
	\end{eqnarray*}
	By Cauchy--Schwarz inequality, we deduce that 
	\[
	\frac{{{{\tilde L}_k}}}{{\| {A\tilde M^{ - 1}} \|^2}} \le 1.
	\]
	Using a similar strategy, we can observe that
	\[
	\frac{{{{\hat L}_k}}}{{\| {A\hat M^{ - 1}} \|^2}} \le 1.
	\]		
		Hence, the following quantity is well-defined:
		\begin{equation}\label{lk}
		\mathcal{L}_k := \sqrt {1 - \frac{{{{\tilde L}_k}}}{{\| {A\tilde M^{ - 1}} \|^2}}} \cdot \sqrt {1 - \frac{{{{\hat L}_k}}}{{\| {A\hat M^{ - 1}} \|^2}}}\,.
		\end{equation}
		From Eqs. \eqref{eq14} and \eqref{eq15}, we can verify that
		\[
		\| r^{(k+1)} \| \le \mathcal{L}_k \|r^{(k)}\| \,.
		\]
		The assumption \eqref{eq11} ensures that the values of  $\tilde{\xi}$ and
		$\hat{\xi}$ cannot be zero simultaneously. This shows $\mathcal{L}_k < 1$ which illustrates that the sequence $\{\|r^{(k)}\|\}_{k=1 }^{\infty}$ is strictly decreasing unless the exact solution is found.
	\end{proof}

Next, we show that  $\|r^{(k)}\|$ actually converges to zero as $k \to \infty$.
First, however, we make two remarks on the previous theorem to address possible breakdowns of the TSTMR method and the worst potential residual norm 
reduction at a given step of the method. 
	
	\begin{rem} \label{rem2.5} {\rm Under the assumptions \eqref{eq11}, given the initial guess $x^{(0)}$, it turns out that either the chain of inequalities
			\[\| {{r^{(0)}}} \| \ge \| {{r^{(\frac{1}{2})}}} \| > \| {{r^{(1)}}} \| \ge \| {{r^{(\frac{3}{2})}}} \| > \| {{r^{(2)}}} \| \ge  \cdots \]
			or
			\[\| {{r^{(0)}}} \| > \| {{r^{(\frac{1}{2})}}} \| \ge \| {{r^{(1)}}} \| > \| {{r^{(\frac{3}{2})}}} \| \ge \| {{r^{(2)}}} \| >  \cdots \]
			holds, provided that $x^{(\frac{1}{2})}$ and $x^{(1)}$ are respectively computed by \eqref{firststepa} and \eqref{firststepb}. Each sets of these inequalities guarantees
			that $\delta ^{(k)}_2$ and $\delta ^{(k+\frac{1}{2})}_2$ are nonzero vectors for $k \ge 1$.
			Note that  if $\delta ^{(k)}_1$ $(\delta ^{(k+\frac{1}{2})}_1)$ is zero then $x^{(k)}$ $(x ^{(k+\frac{1}{2})})$ is the exact solution of \eqref{main}. 
			Now let us consider the case of breakdown for the proposed method in which the matrix  $W_A(\delta ^{(k)}_1,\delta ^{(k)}_2)$ $(W_A(\delta ^{(k+\frac{1}{2})}_1,\delta ^{(k+\frac{1}{2})}_2))$ is singular while  $\delta ^{(k)}_1$ and $\delta ^{(k)}_2$ $(\delta ^{(k+\frac{1}{2})}_1$ and $\delta ^{(k+\frac{1}{2})}_2)$ are nonzero vectors. In this case, we can find the exact solution by Proposition \ref{prop2.1} $($Proposition \ref{prop2.12}$)$.
			Consequently, we conclude that the breakdown of TSTMR method is a lucky breakdown and the method converges to the exact solution of $Ax=b$, if no breakdown happens.
{We should mention that in our numerical tests we have not observed any breakdown, or near-breakdown: in other words, we have never encountered a case in 
which the $2\times 2$ matrix $W_A(\delta ^{(k)}_1,\delta ^{(k)}_2)$ was singular or severely ill-conditioned.}			}
	\end{rem}

	\begin{rem} \label{rem2.6} {\rm 
			It is worth to briefly discuss  the smallest possible reduction at a specific step (say $k$th step), i.e., the case that either $\tilde{\xi}=0$ or $\hat{\xi}=0$, which correspond to the cases $0 \in \mathcal{F}(A\tilde{M}^{-1})$ or $0 \in\mathcal{F}(A\hat{M}^{-1})$, respectively. Considering the sufficient condition $	0 \notin \mathcal{F}(A\tilde{M}^{-1}) ~\cap~  \mathcal{F}(A\hat{M}^{-1})$, without loss of generality, we may assume that $0 \notin \mathcal{F}(A\tilde{M}^{-1})$ which ensures that $\tilde{\xi}\ne 0$. \fb{By substituting $\tilde{L}_k$ from \eqref{Lth} in \eqref{lk}, we can observe that}  the value of $\mathcal{L}_k$ in the proof of previous theorem is bounded above by $\tilde{\mathcal{L}}_k $ given as follows:
			\begin{equation*}
			\tilde{\mathcal{L}}_k := \sqrt {1 - \frac{{{{\tilde \xi}^2}}}{{\| {A\tilde M^{ - 1}} \|^2}}\cdot \frac{\|r^{(k)}\|^2}{\|r^{(k)}\|^2+\|r^{(k)}-r^{(k-1)}\|^2}}.
			\end{equation*}
From \eqref{lk}, it follows that 	$\mathcal{L}_k \le \tilde{\mathcal{L}}_k$. \fb{By the Cauchy--Schwarz inequality, we have $$- \left\langle {{r^{(k)}},{r^{(k - 1)}}} \right\rangle  \le \left\| {{r^{(k)}}} \right\|\left\| {{r^{(k - 1)}}} \right\|$$ which implies 
\[
\|r^{(k)}-r^{(k-1)}\|^2 \le (\|r^{(k)}\|+\|r^{(k-1)}\|)^2.
\]
}
Therefore, in view of the above remark, we have
					\begin{equation*}
						\mathcal{L}_k \le \sqrt {1 - \frac{{{{\tilde \xi}^2}}}{{\| {A\tilde M^{ - 1}} \|^2}}\cdot \frac{\|r^{(k)}\|^2}{\|r^{(k)}\|^2+(\|r^{(k)}\|+\|r^{(k-1)}\|)^2}}
				\end{equation*}
$\tilde{\xi}\ne 0$. For  ease of notation, we define
		\[
		\tilde{\mathcal{C}} _{k}:= { \frac{{{{\tilde \xi}^2}}}{{\| {A\tilde M^{ - 1}} \|^2}}\cdot \frac{\|r^{(k)}\|^2}{\|r^{(k)}\|^2+(\|r^{(k)}\|+\|r^{(k-1)}\|)^2}}
		\]
		and
			\[\tilde{\mathcal{C}} 	:= {\frac{{{{\tilde \xi}^2}}}{{5\| {A\tilde M^{ - 1}} \|^2}}}.\]
		By Theorem \ref{thm2.3}, the sequence of residual norms $\|r^{(k)}\|$ is convergent. Let 	
	$\mathop {\lim }\limits_{k \to \infty } \|r^{(k)}\| =\tau$.
		This shows 
				\[\mathop {\lim }\limits_{k \to \infty } \tilde{\mathcal{C}} _k = 	\tilde{\mathcal{C}}.  \]		
			Using the	Cauchy--Schwarz inequality and the definition of $\tilde{\xi}$, one can observe $ \tilde{\xi} \le \| {A\tilde M^{ - 1}} \|$ which implies $ \tilde{\mathcal{C}} < 1$. The assumption that $\tilde{\xi} \ne 0$ implies  $\tilde{\mathcal{C}}>0$.
				It can be verified that
					\[
				\| r^{(k+1)} \| \le { \mathcal{L}_{k}} \|r^{(k)}\|\le (1-\tilde{\mathcal{C}} _{k})^{1/2}\|r^{(k)}\|.
				\]		  
	Letting $k \to \infty$ in  the above inequalities, we get 
	\[
	\tau \le (1-\tilde{\mathcal{C}})^{1/2} \tau 
	\] 
	which implies $\tau=0$.
	Hence, we have proved the following result.	}

\end{rem}
	\begin{thm}\label{thm2.X}
	Under the assumptions of Theorem \ref{thm2.3}, the method \eqref{method} is convergent.
	\end{thm}

	\begin{rem}\label{rem2.4}{\rm
			As seen, the assumption \eqref{eq11} guarantees the convergence of the sequence of approximate solutions produced by \eqref{method}. Otherwise, we may have
			\begin{equation*}
				\left\langle {{r^{(k)}},A{\delta_1 ^{(k)}}} \right\rangle  = 0 \quad \text{and}\quad 	\left\langle {{r^{(k + \frac{1}{2})}},A{\delta_1 ^{(k + \frac{1}{2})}}} \right\rangle  = 0
			\end{equation*}
			for nonzero vectors $\delta_1^{(k)}$ and $\delta_1^{(k + \frac{1}{2})}$. Note that $A\tilde{M}^{-1}=(\tilde{M}-\tilde{N})\tilde{M}^{-1}=I-\tilde{N}\tilde{M}^{-1}$ and $A\hat{M}^{-1}=(\hat{M}-\hat{N})\hat{M}^{-1}=I-\hat{N}\hat{M}^{-1}$. Hence, the above relations are respectively equivalent to
			\[
			\frac{{\left\langle {{r^{(k)}},\tilde N{{\tilde M}^{ - 1}}r^{(k)}} \right\rangle }}{{\left\langle {{r^{(k)}},{r^{(k)}}} \right\rangle }} = 1\]
			and
			\[
			\frac{{\left\langle {{r^{(k + \frac{1}{2})}},\hat N{{\hat M}^{ - 1}}r^{(k + \frac{1}{2})}} \right\rangle }}{{\left\langle {{r^{(k + \frac{1}{2})}},{r^{(k + \frac{1}{2})}}} \right\rangle }} =1.
			\]
			Consequently, the assumption \eqref{eq11} is equivalent to
			$$	1 \notin \mathcal{F}(\tilde N{{\tilde M}^{ - 1}}) ~\cap~  \mathcal{F}(\hat N{{\hat M}^{ - 1}}).$$
			The above condition holds, if  either
			$\| \tilde{M}^{-1}\tilde{N}\|<1$ or $\| \hat{M}^{-1}\hat{N}\|<1$.}
	\end{rem}

	We conclude this section by commenting that no explicit formula is available for determining the optimum value of the parameter $\alpha$ in the MRHSS method.
	The best value of $\alpha$ is problem-dependent and is usually determined experimentally, limiting the effectiveness of the method, see the numerical experiments in \cite{Yang2,Yang}. In contrast, our implementation of the TSTMR method does not need any free parameters, see Subsection \ref{sub4.1} for more details.

	\section{TSTMR for discrete ill-posed problems}\label{sec3}
	
	In this section, we apply the proposed method 
	to find approximations to the (least-squares) solutions of  linear systems of equations
	\begin{equation}\label{eq1}
		Af=g\,,
	\end{equation}
	where $A\in \mathbb{R}^{m\times n}$, with no restrictions on $m$ and $n$. Such systems  may arise from the discretization of ill-posed problems. Examples include the discretization of inverse problems, such as  image restoration problems, and 
	Fredholm integral equations of the first kind, see \cite{Cui,Hansen,Lv} and the reference therein for more details. In these applications, the right-hand side is typically contaminated by an error (or noise) vector
	$e\in \mathbb{R}^{n}$, i.e., $g=\tilde{g}+e$
	where the vector $\tilde{g}$ represents the unknown, noise-free right-hand side, and the goal 
	  is to find acceptable approximations to the (inaccessible) solution
	of the 
	linear system of equations (or least-squares problem)
$$
		Af=\tilde{g}.
$$
To deal with the ill-posed nature of the problem, a common strategy is to use
	 Tikhonov regularization, which consists of replacing the original problem by the following minimization problem:
	\begin{equation}\label{eq1t}
		\mathop {\min }\limits_f \left\{ {\left\| {Af - g} \right\|^2 + {\mu ^2}\left\|
			Lf \right\|^2} \right\} \,.
	\end{equation}
Here $L$ is the regularization matrix, which is typically chosen to
	be either the identity matrix or a discrete approximation of the derivative operator. In addition,
	the nonnegative constant $\mu$ is the regularization parameter, which is generally small (relative to the data).
	Throughout this paper, we only consider the case that  $L=I$.
	
	In the sequel, the minimization problem \eqref{eq1t} is first reformulated into a linear system of equations and 
	some of the possible solution methods are reviewed. 
	Then, the proposed TSTMR method is adapted to solve a two-by-two augmented block linear system of equations associated with Eq. \eqref{eq1t}.
	
	\subsection{Problem reformulation}
	
	It is well-known that the regularized problem \eqref{eq1t} (with $L=I$) is mathematically equivalent to the following system of (regularized) normal equations:
	\begin{equation}\label{eq1n}
		(A^TA+\mu^2I)f=A^Tg.
	\end{equation}

	Evidently, Eq. \eqref{eq1n}  is equivalent to the following block linear system (e.g., see \cite{Lv})
	\begin{equation}\label{eq20}
		\left[ {\begin{array}{*{20}{c}}
				I&A\\
				{ - {A^T}}&{{\mu ^2}I}
		\end{array}} \right]\left[ {\begin{array}{*{20}{c}}
				e\\
				f
		\end{array}} \right] = \left[ {\begin{array}{*{20}{c}}
				g\\
				0
		\end{array}} \right]	
	\end{equation}
	where $e=g-Af$. In the sequel, for 
	notational simplicity, we set\footnote[2]{We emphasize that the block matrix $K$ in \eqref{Kmatrix} is not  explicitly  formed in practice.}
	\begin{equation}\label{Kmatrix}
		K = \left[ {\begin{array}{*{20}{c}}
				I&A\\
				{ - {A^T}}&{{\mu ^2}I}
		\end{array}} \right], \quad x=\left[ {\begin{array}{*{20}{c}}
				e\\
				f
		\end{array}} \right], \quad b=\left[ {\begin{array}{*{20}{c}}
				g\\
				0
		\end{array}} \right].
	\end{equation}
	
The Hermitian and skew-Hermitian splitting  of $K$ takes the following form:
	\begin{equation}\label{HSS}
		K = \mathcal{H}(K) + \mathcal{S}(K) = \left[ {\begin{array}{*{20}{c}}
				I&0\\
				0&{{\mu ^2}I}
		\end{array}} \right] + \left[ {\begin{array}{*{20}{c}}
				0&A\\
				{ - {A^T}}&0
		\end{array}} \right].
	\end{equation}
	It is immediate to see that $\mathcal{H}(K) \succ 0$ and one can apply the HSS method;  the reader is referred to \cite{Bai,Benzi2004,Benzi2006} for more details.
	Lv et al. \cite{Lv} proposed
	a special HSS (SHSS) iterative method by substituting $\alpha= 1$ into the second step of the HSS method. More precisely, the SHSS iterations produce the approximate solutions to  \eqref{eq20} as follows:
	\begin{equation*}
		\begin{cases}
			(\alpha I + \mathcal{H}(K) ) x^{(k+\frac{1}{2})}=(\alpha I -\mathcal{S}(K) ) x^{(k)}+b\\
			( I +\mathcal{S}(K) ) x^{(k+1)} = ( I - \mathcal{H}(K) ) x^{(k+\frac{1}{2})}+b
		\end{cases}\quad 	(k=0,1,2,\ldots)
	\end{equation*}
	for a given initial  guess $x^{(0)}=[e^{(0)};f^{(0)}]$ and $\alpha > 0$.
	In order to further improve the performance of the SHSS method,  Cui et al. \cite{Cui}
	established the modified SHSS  (MSHSS) method. The  MSHSS method constructs  approximate solutions of \eqref{eq20} using the following two steps:
	\begin{equation}\label{eq23}
		\begin{cases}
			(\alpha I + \mathcal{H}(K) ) x^{(k+\frac{1}{2})}=(\alpha I -\mathcal{S}(K) ) x^{(k)}+b\\
			( \Omega +\mathcal{S}(K) ) x^{(k+1)} = ( \Omega - \mathcal{H}(K) ) x^{(k+\frac{1}{2})}+b
		\end{cases}\quad 	(k=0,1,2,\ldots)
	\end{equation}
	where $
	\Omega  = \left[ {\begin{array}{*{20}{c}}
			I&0\\
			0&{\gamma I}
	\end{array}} \right]$ with prescribed  $\gamma >0 $ ($\gamma \ne \mu^2$), here the initial guess $x^{(0)}=[e^{(0)};f^{(0)}]$ and $\alpha > 0$ are given.
	
	It has been observed that the SHSS method outperforms the standard HSS method for solving $Kx=b$ where the matrix $K$ is given by \eqref{Kmatrix}, see \cite{Cui,Lv}  for further details.
	In addition, the MSHSS method is superior to the SHSS method according to the numerical experiments reported in \cite{Cui}. Therefore, in Example \ref{ex2} below, we only 
	show the results comparing the proposed TSTMR approach with the MSHSS method.
	
	\subsection{TSTMR for solving the regularized problem}\label{sub3.1}
	
	It is immediate to observe that the first and second steps of iterative method \eqref{eq23} correspond to the
	following splitting of the matrix $K$, respectively,
	\begin{equation}\label{eq24}
		K  =(\alpha I + \mathcal{H}(K) ) - \left[ {\begin{array}{*{20}{c}}
				\alpha I &-A\\
				A^T &{\alpha I}
		\end{array}} \right]\quad
		\text{and} \quad
		K =( \Omega +\mathcal{S}(K)) - \left[ {\begin{array}{*{20}{c}}
				0&0\\
				0&{\epsilon I}
		\end{array}} \right]\end{equation}
	where $\epsilon:=(\gamma  - {\mu ^2})$.
	In this subsection, we apply the TSTMR approach to solve the linear system of equations \eqref{eq20}.
	To this end, we set $\alpha=0$, choose a suitable value for $\gamma$ and  apply the TSTMR method in conjunction with the   splittings \eqref{eq24}. 
	The appropriate value of $\gamma$ is determined such that the
	following condition holds:
	\begin{equation}\label{eq112}
		0 \notin \mathcal{F}(K\mathcal{H}(K)^{-1}) ~\cap~  \mathcal{F}(K(\Omega +\mathcal{S}(K))^{-1}),
	\end{equation}
	which guarantees the convergence of the corresponding TSTMR method by Theorem \ref{thm2.X}. To do so, we first need to present the following proposition.  
	
	\begin{prop}\label{prop3.1}
		Let $S=\gamma I + A^TA$ for a given positive constant $\gamma$. Then $\|S^{-1}A^T\| \le \frac{1}{2\sqrt{\gamma}}$.
	\end{prop}
	
	\begin{proof}
		Let $\lambda$ be an arbitrary eigenvalue of $A^TA$, i.e., there exists a nonzero vector $w$ such that
		\[
		A^TA w =\lambda w.
		\]	
		It is obvious that $S^{-1}w=(\gamma I+ A^TA)^{-1}  w =(\gamma+\lambda)^{-1} w.$ Now, it is not difficult to verify that
		\[
		S^{-1}(A^TA)S^{-1} w = \frac{\lambda}{(\gamma+\lambda)^2}w.
		\]
		Evidently, we have
		\[
		\mathop {\max }\limits_{\lambda  \in \sigma ({A^T}A)} \frac{{\sqrt \lambda  }}{{{{(\gamma  + \lambda )}^{}}}} \le \mathop {\max }\limits_{\lambda  \ge 0} \frac{{\sqrt \lambda  }}{{{{(\gamma  + \lambda )}^{}}}}.
		\]
		It can be seen that the function
		$g(x) := \frac{{\sqrt x }}{{(\gamma  + x)}}$ for $x\ge 0$ takes its maximum on $x=\gamma$.	Therefore, we have
		\[
		\|S^{-1}A^T\| =\sqrt{\rho(S^{-1}A^T(S^{-1}A^T)^T)} \le g(\gamma)
		\]
		which completes the proof.
	\end{proof}
	
	Now we establish a theorem from which we can conclude that the condition
	$$
	0 \notin \mathcal{F}(K(\Omega +\mathcal{S}(K))^{-1})
	$$
	holds for certain values of $\gamma$. Obviously, the above condition implies
	\eqref{eq112}.

	\begin{thm}\label{th3.3}
		
		Let the parameter $\gamma$ be chosen such that $\gamma > \mu^2$ where $\mu$ is the given nonnegative parameter in \eqref{eq1n}. Then the real and imaginary parts of $\mathcal{F}(K(\Omega +\mathcal{S}(K))^{-1}) $  satisfy
		\begin{eqnarray}
			\mathcal{R}(K(\Omega +\mathcal{S}(K))^{-1})  &\subset& \left({\frac{{{\mu ^2} + {\lambda _{\min }}({A^T}A)}}{{\gamma  + {\lambda _{\min }}({A^T}A)}} - \bar \eta ,1 + \bar \eta} \right)\label{eq26n} \\
			\mathcal{I}(K(\Omega +\mathcal{S}(K))^{-1}) &\subset& \left(- \frac{1}{2\sqrt{\gamma}} , \frac{1}{2\sqrt{\gamma}} \right)	\label{eq26nn}
		\end{eqnarray}
		where 
		\begin{equation}\label{eq25}
			\mathcal{R}(K(\Omega +\mathcal{S}(K))^{-1}):=		\left\{  \text{\rm Re} (z) ~ | ~ z =\frac{\left\langle {K(\Omega +\mathcal{S}(K))^{-1}y,y} \right\rangle}{\left\langle {y,y} \right\rangle} ~\text{for}~  0\ne y\in \mathbb{C}^{n}  \right\},
		\end{equation}
		\begin{equation}\label{eq25i}
			\mathcal{I}(K(\Omega +\mathcal{S}(K))^{-1}):=		\left\{  {\rm Im} (z) ~ | ~ z =\frac{\left\langle {K(\Omega +\mathcal{S}(K))^{-1}y,y} \right\rangle}{\left\langle {y,y} \right\rangle} ~\text{for}~  0\ne y\in \mathbb{C}^{n}  \right\},
		\end{equation}
		and $\bar \eta = \frac{{\gamma  - {\mu ^2}}}{{2\sqrt \gamma  }}.$ 
	\end{thm}

	\begin{proof}
		For simplicity, we set $\epsilon=(\gamma  - {\mu ^2})$.	Considering \eqref{eq24} and using straightforward computations, one can derive
		\begin{equation*}
			K(\Omega +\mathcal{S}(K))^{-1}  =I-\epsilon \left[ {\begin{array}{*{20}{c}}
					0 &0\\
					S^{-1}A^T &{S^{-1}}
			\end{array}} \right]
		\end{equation*}
		where $S=\gamma I + A^TA$. For an arbitrary nonzero vector $x=[x_1;x_2]$, we have
		\begin{align}
			\left\langle {K(\Omega +\mathcal{S}(K))^{-1}x,x} \right\rangle & = \left\langle {x,x} \right\rangle  -\epsilon  (x_2^*S^{-1}A^Tx_1+x_2^*S^{-1}x_2).\label{eq26}
		\end{align}
		From  the above relation, it turns out that
		\begin{equation}\label{eq33}
			\frac{\left\langle {K(\Omega +\mathcal{S}(K))^{-1}x,x} \right\rangle}{\left\langle {x,x} \right\rangle}  =   1 -\epsilon\,  \frac{x_2^*S^{-1}A^Tx_1+x_2^*S^{-1}x_2}{\left\langle {x,x} \right\rangle}.
		\end{equation}
		
		If $x_1$ is a zero vector, then
		\[
		\frac{\left\langle {K(\Omega +\mathcal{S}(K))^{-1}x,x} \right\rangle}{\left\langle {x,x} \right\rangle}  =   1 -\epsilon\,  \frac{x_2^*S^{-1}x_2}{\left\langle {x,x} \right\rangle}.
		\]
		Therefore, the right-hand side is real and bounded as follows:
		\begin{eqnarray*}
			\frac{{{\mu ^2} + {\lambda _{\min }}({A^T}A)}}{{\gamma  + {\lambda _{\min }}({A^T}A)}}= 1- \frac{\epsilon}{\gamma + \lambda_{\min} (A^TA)} &\le&  1 -\epsilon \,	\frac{x_2^*S^{-1}x_2}{\left\langle {x_2,x_2} \right\rangle}\\
			& =& \frac{\left\langle {K(\Omega +\mathcal{S}(K))^{-1}x,x} \right\rangle}{\left\langle {x,x} \right\rangle} \le 1.
		\end{eqnarray*}
		When $x_2$ is a zero vector, we simply obtain
		\[
		\frac{\left\langle {K(\Omega +\mathcal{S}(K))^{-1}x,x} \right\rangle}{\left\langle {x,x} \right\rangle}  =   1.
		\]
		As a result, if either $x_1$ or $x_2$ is zero, then the value of the left-hand side in \eqref{eq33} is real and bounded by $\frac{{{\mu ^2} + {\lambda _{\min }}({A^T}A)}}{{\gamma  + {\lambda _{\min }}({A^T}A)}} $ (one) from below (above). In the rest of proof, we assume that $x_1$ and $x_2$ are both nonzero vectors. Without loss of generality, we may assume that $0<\|x_i\|<1$ for $i=1,2$. Evidently, we have
		\begin{eqnarray}
			\text{Re}\left(	\frac{\left\langle {K(\Omega +\mathcal{S}(K))^{-1}x,x} \right\rangle}{\left\langle {x,x} \right\rangle}  \right) &=&1 -\epsilon \left(	 \frac{x_2^*S^{-1}x_2}{\left\langle {x,x} \right\rangle}+ 	\frac{\text{Re}\left(x_2^*S^{-1}A^Tx_1 \right) }{\left\langle {x,x} \right\rangle}\right) \label{eq271}\\
			\text{Im}\left(	\frac{\left\langle {K(\Omega +\mathcal{S}(K))^{-1}x,x} \right\rangle}{\left\langle {x,x} \right\rangle}  \right) &=& -\epsilon \, \frac{\text{Im}\left(x_2^*S^{-1}A^Tx_1 \right) }{\left\langle {x,x} \right\rangle}. \label{eq281}
		\end{eqnarray}
		Using the Cauchy--Schwarz inequality and assuming, without loss of generality, that $\|x_2\| \le \|x_1\| < 1$, we conclude
		\[
		\frac{|x_2^*S^{-1}A^Tx_1|}{\left\langle {x,x} \right\rangle} < \frac{|x_2^*S^{-1}A^Tx_1|}{\left\langle {x_1,x_1} \right\rangle} \le \frac{\|S^{-1}A^Tx_1\| }{\|x_1\|}.
		\]
		
		Now, by Proposition \ref{prop3.1}, one can observe that
		\begin{equation*}
			\frac{|x_2^*S^{-1}A^Tx_1|}{\left\langle {x,x} \right\rangle} < \frac{1}{2\sqrt{\gamma}}.
		\end{equation*}
		By the above inequality, we can conclude intervals \eqref{eq26n} and \eqref{eq26nn} associated with \eqref{eq271} and \eqref{eq281}, respectively.
	\end{proof}
	
	\begin{rem}\label{th3.2}
		{\rm	Let the parameters $\mu \ge 0$ and $\gamma>0$ be given. From Theorem \ref{th3.3}, it is immediate to verify that 	$0 \notin \mathcal{F}(K(\Omega +\mathcal{S}(K))^{-1})$ when
			\begin{equation}\label{eq28n}
				0 < (\gamma  - {\mu ^2}) <  \frac{2\sqrt{\gamma}(\mu^2+\lambda_{\min}(A^TA))}{\gamma+\lambda_{\min}(A^TA)}
			\end{equation}
		Indeed,  for any $z\in \mathcal{F}(K(\Omega +\mathcal{S}(K))^{-1}),$ we have ${\rm Re}(z)>0$.}
	\end{rem}
	
	The above remark can be exploited for determining the suitable value of $\gamma$. \fb{One can verify that 
\begin{equation*}	
	\frac{2\mu^2}{\sqrt{\gamma}} <  \frac{2\sqrt{\gamma}(\mu^2+\lambda_{\min}(A^TA))}{\gamma+\lambda_{\min}(A^TA)},
\end{equation*}
provided that $\mu^2 < \gamma$. The preceding bound is sharp since $\lambda_{\min}(A^TA)$ is typically very small in the case of  ill-posed problems. 
Hence, the value of $\gamma$ can be determined such that 
\begin{equation}\label{1}
0<	\gamma  - {\mu ^2} < \frac{2\mu^2}{\sqrt{\gamma}}. 
\end{equation}	
}	
	\fb{By Descartes' rule of signs, the cubic polynomial $t^3-\mu^2 t-2\mu^2=0$ has a unique positive root.
		Hence, we conclude that there exists $\bar{t}>0$ such that
		\[
	\bar{t}(\bar{t}^2-\mu^2)=2\mu^2,
		\]
which ensures that $\bar{t}^2> \mu^2$.
Therefore, setting $\gamma^*=\bar{t}^2$,} one can see that there is an open interval $I=(\mu^2, \gamma^*)$, independent of $\lambda_{\min} (A^TA)$,  such that ${\rm Re}(z)>0$ for all $\gamma \in I$. To be more precise, the constant $\gamma^*$ is the (unique) positive solution of the equation $\sqrt{\gamma} (\gamma - \mu^2) = 2\mu^2$.
	Our experimental results show that even the largest value of $\gamma$ that satisfies \fb{either \eqref{eq28n} or \eqref{1}} leads to feasible performance of the proposed TSTMR method. 
	
	The following last remark on Theorem \ref{th3.3}  provides a more explicit upper bound for the real part of the FoV of
	$K(\Omega +\mathcal{S}(K))^{-1}$. 
	
	\begin{rem}\label{prop3.4}{\rm
			In addition to the assumptions of Theorem \ref{th3.3}, assume that $\gamma <4$.  Let
			$z_1\in 	\mathcal{R}(K(\Omega +\mathcal{S}(K))^{-1}) $ where  $\mathcal{R}(K(\Omega +\mathcal{S}(K))^{-1})$ is defined by \eqref{eq25}. Notice that  $\sqrt{\gamma} <2$ together with Theorem \ref{th3.3} imply that
			\[
			z_1 \le 1 + \frac{{\gamma  - {\mu ^2}}}{{2\sqrt \gamma  }} \le 1 + \frac{{\gamma  - {\mu ^2}}}{{ \gamma  }}<2.
			\]
		}
	\end{rem}

	\section{Numerical experiments}\label{sec4}
	
	In this section, some numerical results are presented to illustrate the feasibility of  the proposed
	TSTMR solver and to compare its performance with some of the existing methods in literature. All of the numerical computations were carried out on a computer with an Intel Core i7-10750H CPU @ 2.60GHz processor and 16.0GB RAM using MATLAB.R2020b.
	
	We report the total required number of iterations and elapsed CPU time (in seconds) under ``Iter" and ``CPU", respectively.  In the  tables, we also include the relative error
	\[
	\mathrm{Err}:=\frac{{\|x^{(k)}-x^\ast\|}}{{\|x^\ast\|}},
	\]
	where $x^\ast$ and $x^{(k)}$ are receptively the exact solution and  its approximation obtained in the $k$th iterate.
	The reported CPU times and iteration counts (rounded to the nearest integer) in the tables are obtained as \fb{the average of ten runs.}

	For more clarification, the following section is divided into three subsections.
	We first report some comparison results between the TSTMR and MRHSS methods for which the splittings in the TSTMR method correspond to the symmetric and shifted skew-symmetric parts of the coefficient matrix on linear systems arising from a finite difference discretization of some convection-diffusion PDEs. Then, the performance of the proposed method is compared with the flexible GMRES (FGMRES) method (in conjunction with a suitable preconditioner) for determining approximate solutions of linear systems of equations arising from finite element discretization of the coupled Stokes-Darcy flow problem.  The second and third parts deal with finding the solution of \eqref{eq1} corresponding to ill-posed test problems in order to numerically illustrate the performance of  the  variant of the TSTMR method  proposed in Subsection \ref{sub3.1}.  Depending on the examined ill-posed test problems, the  method is compared with the methods proposed in \cite{Chung} or \cite{Cui}.

	\subsection{Experimental results for two well-posed test problems}\label{sub4.1}
	
	In this part, we first consider a test example  from \cite[Example 1]{Yang} in order to compare the performance of TSTMR with MRHSS. In the second example, the proposed method is used for determining an approximate solution of a 3D coupled Stokes-Darcy problem with large jumps in the permeability \cite{Chid} and its performance is compared with FGMRES  used in conjunction with an efficient preconditioner proposed in \cite{BB2022}.

	 For the reported experiments in this subsection,  we terminated the iterations once
   \begin{eqnarray}\label{cer}
  \frac{{\|{b-Ax^{(k)}}\|}}{{\|b\|}}\leq 10^{-8},
   \end{eqnarray}
	or if $k\ge 10000$, where $x^{(k)}$ is the $k$th approximate solution and the initial vector $x^{(0)}$ is taken to be
	zero. The right-hand side $b$ in \eqref{main} corresponds to a random solution vector $x^\ast=\text{rand}(n,1)$.

	\begin{example}\label{eq01}
		{\rm 	Consider the following two-dimensional convection-diffusion equation
			\begin{align}\label{eqcd}
				-\left(\frac{\partial^2u}{\partial x^2}  +\frac{\partial^2u}{\partial y^2}  \right) + a(x,y)\frac{\partial u}{\partial x}+ b(x,y)\frac{\partial u}{\partial y}=&f(x,y),\quad \mathrm{in}\, \Omega,\\
				&u=0,\quad  \mathrm{on}\, \partial\Omega,
			\end{align}
			where $\Omega= [0, 1] \times [0, 1]$. 
			The coefficient functions $a(x, y)$ and $b(x, y)$ are chosen as follows:
			\begin{align*}
				&\mathrm{Case\,I}.\quad a(x,y)= x\sin(x+y),\quad b(x,y)= y\cos(xy);\\
				&\mathrm{Case\,II}.\quad a(x,y)=5  y\exp(xy),\quad b(x,y)=5  x\exp(x+y).
			\end{align*}
			We discretize equation \eqref{eqcd} by using the standard finite-point central difference discretization
			with mesh size $h=1/l$ for different values of $l$ and obtain the linear systems $Ax = b$, where
			$A\in \mathbb{R}^{(l-1)^2\times (l-1)^2}$ where $A$ is a  non-symmetric positive definite matrix.

	\begin{table}[h]
	\centering
	\caption	{Example \ref{eq01}: Numerical results for MRHSS and TSTMR methods}{
		{	
			{	\begin{tabular}{cllllll}
					\toprule&&Case I&&&Case II&\\
					\cmidrule{3-4}\cmidrule{6-7}
					Method	&$l$&$80$&$160$&&$80$&$160$\\
					\midrule
					\multirow{15}[2]{*}{ MRHSS}&$\alpha_1$&0.1551&0.0771&&0.1378&0.0685\\
					&Iter&384& 701&&223&397\\
					&CPU &1.85&19.6&&1.33& 15.6\\
					&Err&$7.0651{\rm e}{-06}$&$ 2.8464{\rm e}{-05}$&&$4.3216{\rm e}{-06}$&$1.7699{\rm e}{-05}$\\
					\cmidrule{2-7}
					&$\alpha_2$&$8.5775{\rm e}{-06}$&$2.1409{\rm e}{-06}$&&$2.2865{\rm e}{-03}$&$5.7392{\rm e}{-04}$\\
					&Iter&5&5&&50&46\\
					&CPU&0.04& 0.16&&0.38&1.71\\
					&Err&$6.2753{\rm e}{-06}$&$3.1246{\rm e}{-06}$&&$5.3823{\rm e}{-07}$&$2.1214{\rm e}{-06}$\\
					\cmidrule{2-7}
					&$\alpha_3$&0.0287&0.0142&&0.0293&0.0143\\
					&Iter&80&128&&52&86\\
					&CPU &0.38& 3.54&&1.29&16.4\\
					&Err&$6.4128{\rm e}{-06}$&$2.7094{\rm e}{-05}$&&$2.9773{\rm e}{-06}$&$1.5767{\rm e}{-05}$\\
					\cmidrule{2-7}
					&$\alpha_{\rm exp}$&0.0002&0.0001&&0.009&0.003\\
					&Iter&5&5&&38&35\\
					&CPU &0.04&0.17&&0.32&1.55\\
					&Err&$2.0237{\rm e}{-07}$&$9.4317{\rm e}{-07}$&&$6.5348{\rm e}{-07}$&$5.1019{\rm e}{-06}$\\
					\midrule
					TSTMR&Iter&5&4&&27&24\\
					&CPU  &0.03&0.11&&0.13&0.70\\
					&Err&$7.0589{\rm e}{-08}$&$4.4946{\rm e}{-06}$&&$4.9233{\rm e}{-07}$&$2.3804{\rm e}{-06}$\\
					\bottomrule
	\end{tabular}}}}
	\label{tab1}%
\end{table}%
					
	 To compare the performances of the TSTMR and MRHSS iterative methods for solving linear system \eqref{main}, we set  $\tilde{M}=\mathcal{H}(A)$ and $\hat{M}=\mathcal{S}(A)+\eta^* I$ in the implementation of  TSTMR method where $\eta^*$ is computed by \eqref{eta}. In this case, we have $\| \hat{M}^{-1}\hat{N}\|<1$ by Proposition \ref{prop2.5} and  by Remark \ref{rem2.4}, the TSTMR method converges to the unique solution of the linear system. Unlike the MRHSS method, TSTMR does not face the difficulty of choosing appropriate parameters with these particular choices  of $\tilde{M}$ and $\hat{M}$. For both Cases I and II, direct solvers are used to solve the subsystems of linear equations appearing in the implementation of the TSTMR  and MRHSS methods. More precisely,  the subsystems with SPD coefficient matrices are solved by using the sparse Cholesky factorization with the symmetric approximate minimum degree (SYMAMD) reordering available in MATLAB. The LU factorization in combination with the same reordering (for $\alpha_1,\alpha_3$) or in combination with column approximate minimum degree (COLAMD) reordering (for $\alpha_2,\alpha_{\rm exp}$) is used for solving the shifted linear systems associated with the skew-symmetric part of $A$. We comment that using COLAMD instead of SYMAMD results in a better CPU time for the MRHSS method when the shift on $\mathcal{S}(A)$ is very small.

In Table \ref{tab1}, the values of $\alpha_{\rm exp}$ and $\alpha_i$ ($i=1,2,3$) for the MRHSS method are chosen analogously to \cite[Tables 1 and 3]{Yang}. Specifically, it was mentioned  there that $\alpha_{\rm exp}$ is the experimentally found optimum value of the parameter. As seen, the value of $\alpha$ affects  the performance of  MRHSS significantly. This makes it crucial to have a practical strategy rather than a trial-and-error approach for finding a suitable value for the parameter in the MRHSS method.  Overall, in view of the reported numerical results and of its parameter-free nature,  it can be seen that the TSTMR method is far superior to MRHSS.		
	}
	\end{example}

In the following example, we use the proposed method for solving a three-by-three  linear systems of equations corresponding to a 3D coupled flow problem \cite{Cai,Chid}. For Example \ref{well2}, the proposed method works with a block triangular splitting of the coefficient matrix instead of its symmetric and skew-symmetric parts.

\begin{example}\label{well2}
{\rm   We consider the following linear system of equations 
	\begin{equation}\label{eqwll2}
		Ax = \left[ {\begin{array}{*{20}c}
				{A_{11} } & {A_{{12} } } & 0  \\
				{ A_{21 } } & {A_{22 } } & {B^T }  \\
				0 & B & 0  \\
		\end{array}} \right]\left[ {\begin{array}{*{20}c}
				{u_1 }  \\
				{u_2 }  \\
				{u_3}  \\
		\end{array}} \right] = \left[ {\begin{array}{*{20}c}
				{b_1 }  \\
				{b_2 }  \\
				{b_3}  \\
		\end{array}} \right] = b,
	\end{equation}
	where $A_{11}$ and $A_{22}$  are both SPD, $A_{21}:=-A_{12}^T$ and $B$ has full row rank.
Here, the linear system of equations \eqref{eqwll2} arises from finite element discretizations of the coupled Stokes-Darcy flow problem examined in \cite[Subsection 5.3]{Chid}. 
	}	
	{\rm
	\begin{table}{
			\centering
			\caption{Example \ref{well2}: Comparison results between proposed method and the FGMRES method in conjunction with AL-based preconditioner}
			{
				\begin{tabular}{@{}lcccccccccccc@{}}
					\toprule
					
					size && \multicolumn{5}{c}{{Proposed method}} & \phantom{abc} &\multicolumn{5}{c}{FGMRES}  \\
					\cmidrule{3-7}\cmidrule{9-13} 
					&& Iter && CPU && Err && Iter && CPU && Err   \\
					\midrule
					1695  && 6  && 0.05 && 9.9871e-07 && 19 && 0.06 && 1.4624e-03  \\
					10809 && 6  && 0.42 && 5.2198e-06 && 17 && 0.76 && 4.0111e-04 \\
					76653 && 6  && 4.39 && 4.6253e-05 && 19 && 7.34 && 2.9103e-03  \\
					576213&& 6  && 79.7 && 1.0826e-06 && 26 && 79.3&& 2.2531e-03\\
					\bottomrule	
				\end{tabular}	\label{tabwell2}
			}
			
	}	
		
	\end{table}

}
\end{example}

	Krylov subspace methods (such as the GMRES method) in conjunction with appropriate preconditioners have been an effective approach to the solution  of the discrete coupled Stokes-Darcy equations, see \cite{BB2022,Chid,Cui} and the references therein.  Noting that in practice the preconditioners must be applied inexactly when the underlying PDE problem is 3D, the numerical experiments reported in  \cite{BB2022} indicate that among all examined inexact variants of preconditioners, the augmented Lagrangian (AL) based preconditioner with IC-CG inner solvers leads to the fastest convergence speed of the FGMRES method in term of total solution times by a large margin comparing to the preconditioners proposed in \cite{Chid,Cui}. Therefore, here, we only report comparison results between the proposed method and the FGMRES method in conjunction with AL-based preconditioner with the most efficient implementation  in \cite{BB2022}.

To apply the proposed method, we consider the splitting $A=\bar{M}-\bar{N}$ where
	\begin{equation}\label{eq1810}
		\bar{M} := \left[ {\begin{array}{*{20}{c}}
				{{A_{11}}}&{{A_{12}}}&0\\
				0&{{A_{22}}}&{{B^T}}\\
				0&B&0
		\end{array}} \right].
	\end{equation}
With a strategy similar to the one used in  \cite{Chid}, one can verify that  $\bar{M}$ is FoV-equivalent\footnote{For more details on concept of FoV-equivalence, we refer the reader to \cite{Loghin}.} to the coefficient matrix $A$ in \eqref{eqwll2} for a certain choice of inner products which implies  $0\notin \mathcal{F}(\bar{M}^{-1}A)$. This motivates us to apply the TSTMR method with $\tilde{M}=\hat{M}=	\bar{M}$. Hence, in this case the TSTMR method reduces to a one-step iterative method.
In practice, we approximate the action of $\bar M^{-1}$ using FGMRES (with a loose stopping residual tolerance $0.05$) preconditioned by
\begin{equation*}
	P = \left[ {\begin{array}{ccc}
			I&0&0\\
			{{{0}}}&I&0\\
			0&B\hat A_{22}^{-1}&I
	\end{array}} \right]\left[ {\begin{array}{ccc}
			{{\hat A_{11}}}&A_{12}&0\\
			{{{0}}}&{{\hat A_{22}}}&{{B^T}}\\
			0&0&-{\rm diag}(M_p)
	\end{array}} \right].
\end{equation*}
Here  $\hat A_{11}$ and
$\hat A_{22}$ are approximations of 
$A_{11}$ and $A_{22}$ obtained via incomplete Cholesky factorizations constructed by MATLAB
function
``\verb|ichol|(., \verb|opts|)" and MATLAB backslash operator ``$\backslash$", 
with {\verb|opts.type| ='\verb|ict|'} and {\verb|opts.droptol| =$\epsilon_i$} where $\epsilon_i$ is equal to  $10^{-3}$ and $10^{-2}$ for $i=1,2$, respectively. 
Also, $M_p$ denotes the mass matrix coming  from the Stokes pressure; see \cite{Elman} for more details.

In Table \ref{tabwell2}, we show the results comparing the proposed method with the fastest  approach in \cite{BB2022}. In term of the accuracy of obtained approximate solutions, the proposed method outperforms the preconditioned FGMRES method. For each individual problem size, the method is competitive with the FGMRES method in term of CPU times for convergence with respect to the stopping criterion \eqref{cer}.

	\subsection{Experimental results for some ill-posed test problems}

	This section is devoted to numerically examining the applicability of the proposed method to solve systems of the form \eqref{eq1n} with $\mu>0$.  
	 Since the coefficient matrix in \eqref{eq1n} is symmetric positive definite, the system can be solved, in principle, by the conjugate gradient method.
	The performance of this method, however, is highly sensitive to the choice of the regularization parameter, and can be quite poor for very small $\mu$.
	In the following, we  solve three test problems from Hansen's package \cite{Hansen1994} and compare the performances of the TSTMR and MSHSS \cite{Cui} methods.  For solving these test problems, we  apply the iterative methods to solve $Kx=b$  given by Eq. \eqref{Kmatrix} in which the value of the regularization parameter $\mu$ is estimated by \fb{generalized cross validation (GCV)  \cite{Golub2}, which is one of the commonly employed approaches for determining the regularization parameter in Tikhonov regularization}. In this part, the iterative methods are terminated once ${{\|{b-K x^{(k)}}\|}}/{{\|b\|}}\leq 10^{-6}$ or when the maximum number of 100 iterations is reached. Here, the initial vector is zero and $x^{(k)}$ refers to the $k$th approximate solution as before.
	
	For solving the following example, in the TSTMR method, we set $\tilde{M}=\mathcal{H}(K)$ and
		\begin{equation}\label{mhat}
		\hat{M}	=\left[ {\begin{array}{*{20}{c}}
				I&A\\
				{ - {A^T}}&{\gamma I}
		\end{array}} \right]\,,
	\end{equation}
	with $\gamma > \mu^2$.

	\begin{example}\label{ex2}
		{\rm Consider the block system \eqref{eq20} with $g$ contaminated by noise such that $g = \tilde{g} + 0.01 \times  \text{rand}(\text{size}(\tilde{g} ))$ where the matrix $A\in \mathbb{R}^{n\times n}$ and the vector $\tilde{g}\in \mathbb{R}^{n}$ are constructed with MATLAB function
			$[A,f,\tilde{g}]={\rm \bf Problem}(n)$ where ${\rm \bf Problem}(n)$ is set to be {\bf foxgood}($n$), {\bf gravity}($n$), and {\bf phillips}($n$), respectively. 
			The condition number and the number of nonzero ($nnz$) entries of $A$,  associated with these test problems, are summarized in Table \ref{tabp}. Notice
			that these linear systems are dense.
			}
	\end{example}

	\begin{table}[h]
		\centering
		\caption{Examples \ref{ex2}: Condition number and  $nnz$ of the matrix $A$}
		{{
				\begin{tabular}{ccccccccccccc}
					\toprule
					&\multicolumn{5}{c}{Condition number}&&&\multicolumn{5}{c}{nnz} \\
					\cmidrule{2-6}\cmidrule{9-13}   $n$&\textsf{foxgood}($n$)&&\textsf{gravity}($n$)&&\textsf{phillips}($n$)	&&&\textsf{foxgood}($n$)&&\textsf{gravity}($n$)&&\textsf{phillips}($n$)\\
					\midrule
					900 &$6.3540{\rm e}{+20}$&&$9.7128{\rm e}{+20}$&&$1.7316{\rm e}{+10}$&&&810000&&810000&&355050 \\
					2500 &$1.7487{\rm e}{+21}$&&$1.1300{\rm e}{+21}$&&$7.5805{\rm e}{+11}$&&&6250000&&6250000 &&2736250\\
					4900 &$8.2841{\rm e}{+21}$&&$ 1.4201{\rm e}{+21}$&&$1.3578{\rm e}{+14}$&&& 24010000&&24010000 &&10508050\\
					\bottomrule
		\end{tabular}}}%
		\label{tabp}%
		
		\end{table}%
		\begin{table}[h]
		\centering
		\caption{Example \ref{ex2}: Numerical results for Experiment I with GMRES as inner solver}
		{{
				\begin{tabular}{lccccccccccc}
					\toprule
					&	&\multicolumn{4}{c}{TSTMR}&&&\multicolumn{4}{c}{MSHSS} \\
					\cmidrule{3-6}\cmidrule{8-12}   Problem&& Iter (CPU)&Err&&Res&&& Iter (CPU)&Err&&Res \\
					\midrule
					\textsf{foxgood}(900)&&4(0.11)&0.0468&&0.0111&&&100(2.51)&0.0480&&0.0241\\
					\textsf{gravity}(900)&$$&6(0.36)&0.0106&&0.0011&&&100(5.86)&0.0126&&0.0011\\
					\textsf{phillips}(900)&&5(0.46)&0.0353&&0.0098&&&23(1.68)&0.0358&&0.0098 \\
					\midrule
					\textsf{foxgood}(2500)&&4(0.83)&0.0490&& 0.0112&&&100(19.7)&0.0533&&0.0243 \\
					\textsf{gravity}(2500)&$$&5(2.76)&0.0106&& 0.0011&&&100(48.5)&0.0106&&0.0011 \\
					\textsf{phillips}(2500)&&6(3.56)&0.0414&&0.0162&&&23(13.2)&0.0417&&0.0163\\
					\midrule
					\textsf{foxgood}(4900)&&3(2.44)&0.0424&&0.0112&&&100(72.3)&0.0466&&0.02505 \\
					\textsf{gravity}(4900)&$$&5(9.78)&0.0106&&0.0011&&&100(185)& 0.0081&&0.0011 \\
					\textsf{phillips}(4900)&&7(17.4)&0.0719&& 0.0229&&&42(93.0)&0.0730&&0.0229\\
					\bottomrule
		\end{tabular}}}%
		\label{tab21}%
	\end{table}%

	
	\begin{table}[h]
		\centering
		\caption{Example \ref{ex2}: Numerical results for Experiment II with GMRES as inner solver}
		{{
				\begin{tabular}{lccccccccccc}
					\toprule
					&	&\multicolumn{4}{c}{TSTMR}&&&\multicolumn{4}{c}{MSHSS} \\
					\cmidrule{3-6}\cmidrule{8-12}   Problem&& Iter (CPU)&Err&&Res&&& Iter (CPU)&Err&&Res \\
					\midrule
					\textsf{foxgood}(900)&&3(0.09)&0.0340&&0.0111&&&100(2.76)&0.0360&&0.0114\\
					\textsf{gravity}(900)&&2(0.14)&0.0095&&0.0010&&&21(1.48)&0.0101&&0.0011\\
					\textsf{phillips}(900)&&3(0.09)&0.0470&&0.0112&&&100(2.78)&0.0574&&0.0118 \\
					\midrule
					\textsf{foxgood}(2500)&&2(0.46)&0.0427&&0.0111&&&100(21.5)&0.0498&&0.0118 \\
					\textsf{gravity}(2500)&&2(1.05)&0.0100&&0.0010&&&54(30.2)&0.0196&&0.0011 \\
					\textsf{phillips}(2500)&&3(1.80)&0.0458&&0.0163&&&5(3.04)&0.0460&&0.0163\\
					\midrule
					\textsf{foxgood}(4900)&&3(1.90)&0.0413&&0.0111&&&100(80.7)&0.0514&&0.0121 \\
					\textsf{gravity}(4900)&&2(3.95)&0.0100&&0.0011&&&47(99.1)&0.0091&&0.0011\\
					\textsf{phillips}(4900)&&3(8.44)&0.0845&&0.0228&&&9(23.6)&0.0865&& 0.0228\\
					\bottomrule
		\end{tabular}}}%
		\label{tab31}%
		\end{table}%

		\begin{table}[h]
		\centering
		\caption{Example \ref{ex2}: Numerical results for CGW and inexact TSTMR  methods ($\gamma=\mu^2 +0.01$ and max$_{\rm itcg}=20$)}
		{{
				\begin{tabular}{lccccccccccc}
					\toprule
					&	&\multicolumn{4}{c}{TSTMR}&&&\multicolumn{4}{c}{CGW} \\
					\cmidrule{3-6}\cmidrule{8-12}   Problem&& Iter (CPU)&Err&&Res&&& Iter (CPU)&Err&&Res \\
					\midrule
					\textsf{foxgood}(900)&&6(0.03)&0.0414&&0.0111&&&9(0.03)&0.0425&&0.0111 \\
					\textsf{gravity}(900)&&5(0.03)&0.0110 &&0.0011&&&63(0.18)&0.0602&&0.0011\\
					\textsf{phillips}(900)&&6(0.03)&0.0339&&0.0098&&&49(0.14)&0.0346&&0.0098 \\
					\midrule
					\textsf{foxgood}(2500)&& 7(0.24)&0.0415&&0.0111&&&10(0.22)&0.1116&&0.0111  \\
					\textsf{gravity}(2500)&& 5(0.33)&0.0108&&0.0011&&&46(1.02)&0.0093&&0.0011\\
					\textsf{phillips}(2500)&&7(0.41)&0.0484&&0.0163&&&53(1.16)&0.0493&&0.0163\\
					\midrule
					\textsf{foxgood}(4900)&&3(0.43)& 0.0421&&0.0111&&&7(0.60)&0.0736&&0.0111 \\
					\textsf{gravity}(4900)&&5(1.37)&0.0103&&0.0011&&&104(8.84)&0.1565&& 0.0018\\
					\textsf{phillips}(4900)&& 8(2.04)&0.0677&&0.0229&&&54(4.47)&0.0688&&0.0229\\
					\bottomrule
		\end{tabular}}}%
		\label{tab31inexact}%
	\end{table}%

	\begin{table}[h]
		\centering
		\caption{Example \ref{ex2}: Information on real part of $\mathcal{F}(K(\Omega +\mathcal{S}(K))^{-1})$  for Experiment I}
		{{
				\begin{tabular}{lccccccccc}
					\toprule
					&	&\multicolumn{2}{c}{$n=900$}&&\multicolumn{2}{c}{$n=2500$}&&\multicolumn{2}{c}{$n=4900$} \\
					\cmidrule{3-4}\cmidrule{6-7}\cmidrule{9-10}   Problem&&  \eqref{eq28n} &Interval \eqref{eq26n}&&\eqref{eq28n}&Interval \eqref{eq26n}&&  \eqref{eq28n} &Interval \eqref{eq26n} \\
					\midrule
					\textsf{foxgood}($n$)&& \xmark  &(-0.0477,1.0499)&&\xmark &(-0.0488,1.0500)&&\xmark &(-0.0500,1.0500)\\
					\textsf{gravity}($n$)&& \xmark  &(-0.0330,1.0496)&&\xmark &(-0.0385,1.0497)&&\xmark &(-0.0455,1.0499)\\
					\textsf{phillips}($n$)&& $\checkmark$ &(0.1746,1.0442)&&$\checkmark$&(0.13014,1.0454)&&$\checkmark$&(0.1192,1.0457)\\
					\bottomrule
		\end{tabular}}}%
		\label{tabbound1}%
	\end{table}%

		\begin{table}[h]
		\centering
		\caption{Example \ref{ex2}: Information on real part of $\mathcal{F}(K(\Omega +\mathcal{S}(K))^{-1})$  for Experiment II}
		{{
				\begin{tabular}{lccccccccc}
					\toprule
					&	&\multicolumn{2}{c}{$n=900$}&&\multicolumn{2}{c}{$n=2500$}&&\multicolumn{2}{c}{$n=4900$} \\
					\cmidrule{3-4}\cmidrule{6-7}\cmidrule{9-10}   Problem&&  \eqref{eq28n} &Interval \eqref{eq26n}&&\eqref{eq28n}&Interval \eqref{eq26n}&&  \eqref{eq28n} &Interval \eqref{eq26n} \\
					\midrule
					\textsf{foxgood}($n$)&& $\checkmark$ &(0.0078,1.0156)&& \xmark &(-0.0141,1.0158)  && \xmark &(-0.0152,1.0158)\\
					\textsf{gravity}($n$)&& $\checkmark$ &(0.1475,1.0145)&& $\checkmark$ &(0.0909,1.0150)&& $\checkmark$ &(0.0499,1.0153)\\
					\textsf{phillips}($n$)&&$\checkmark$&(0.6631,1.0091)&& $\checkmark$ &(0.6237,1.0096)&&$\checkmark$&(0.6653,1.0090) \\
					\bottomrule
		\end{tabular}}}%
		\label{tabbound2}%
	\end{table}%

	In \cite[Theorem 3.2]{Cui}, it is proved that the optimum values of $\alpha$ and $\gamma$ in the MSHSS method (see \eqref{eq23}) are
$\gamma^* \to \mu^2+$  (meaning that $\gamma^* > \mu^2$ should be chosen as close as possible to $\mu^2$)
and
\begin{equation}\label{eq40}
	{\alpha ^*} = \mathord{\buildrel{\lower3pt\hbox{$\scriptscriptstyle\frown$}}
		\over \alpha } ({\gamma ^*}) = \frac{{{\gamma ^*}(\sigma _1^2 + \sigma _n^2) + 2\sigma _1^2\sigma _n^2}}{{2{\gamma ^*} + \sigma _1^2 + \sigma _n^2}}
\end{equation}
where $\sigma_1$ and $\sigma_n$ stand for the extreme singular values of $A$. 
Note that when $\sigma_1 = 1$ and $\sigma_n = 0$ (as is reasonable to assume
for discrete ill-posed problems) this expression reduces to $\alpha^* = \mu^2/(2\mu^2 +1)$ in the limit $\gamma^* \to \mu^2$.

We have observed that in practice, the value of $\gamma^*$ in \eqref{eq40} may have a substantial effect on determining the optimum value of $\alpha$ ($\alpha^*$).  
As a matter of fact, the above formula does not provide a suitable approximation
for the optimum value of $\alpha$ when $\gamma$ is not sufficiently close to $\mu^2$.
In order to show this,  we report the results associated with two different values for $\gamma$, i.e., $\gamma=\mu^
2 +0.01$ ({Experiment I}) and  $\gamma=\mu^
2 +0.001$ ({Experiment II}).
To solve the shifted skew-symmetric subsystem inside each iteration of TSTMR and MSHSS, we use  GMRES with no restarting, stopping the iterations once the relative residual 2--norm has been reduced below $10^{-6}$. 
In Tables \ref{tab21} and \ref{tab31} we report the numerical results obtained for Experiments I and II, respectively.
In addition, we also report the value of the norms of the residual vectors associated with the computed approximations to the solution of \eqref{eq1}, i.e.,
\[\text{Res}:= \|\tilde{g} -Af^{(k)}\|/\|\tilde{g}\|,\]
where $x^{(k)}=[e^{(k)};f^{(k)}]$.
Clearly, the proposed TSTMR method outperforms the  MSHSS method on these examples.   We also found the TSTMR method to be much more robust than the CG method
applied to \eqref{eqn1} with respect to the value of $\mu$.

In the following numerical tests, for a more efficient implementation, we apply the inexact version of the TSTMR method to solve \eqref{eq20}. 
To approximate the action of $\hat{M}^{-1}$, we solve the linear systems of equations 
	\[
	\hat{M} [x_1;x_2]=[b_1;b_2]
	\] 
	inexactly, using an inner iteration combined with a loose stopping tolerance.
	To this end, first, we consider the following equivalent linear system of equations
	\[\left[ {\begin{array}{*{20}{c}}
			I&B\\
			{ - {B^T}}&I
	\end{array}} \right]\left[ {\begin{array}{*{20}{c}}
			{{ x_1}}\\
			{{\tilde  x_2}}
	\end{array}} \right] = \left[ {\begin{array}{*{20}{c}}
			{{ b_1}}\\
			{{\tilde  b_2}}
	\end{array}} \right]\]
	where $B=\frac{1}{\sqrt{\gamma}}  A$, $\tilde x_2=\sqrt{\gamma}x_2$ and $\tilde b_2=\frac{1}{\sqrt{\gamma}} b_2$. Then, the approximate solution of the above linear systems of equations is determined in two steps:
	\begin{itemize}
		\item To find $\tilde x_2$, the linear system of equations $(I+B^TB)\tilde x_2=\tilde b_2 + B^T b_1$ is solved by the Conjugate Gradient (CG) method  with a relative residual tolerance of $10^{-2}$  and  a prescribed maximum allowed number of iterations reported above tables under ``max$_{\rm itcg}$". We 
		emphasize that the matrix $I+B^TB$ is not formed explicitly.
		\item We set $ x_1=b_1-B \tilde x_2$.
	\end{itemize}
	This procedure yields an approximate solution for $\hat{M} [x_1;x_2]=[b_1;b_2]$.
	
	 For the sake of comparison, in Example \ref{ex2}, the performance of Concus, Golub, and Widlund (CGW) method \cite[Section 9.6]{Saad} is also reported for solving \eqref{eq20}. This method is in principle well-suited for systems with coefficient matrices of
	the form ``diagonal plus skew-symmetric," as in \eqref{eq20}. 
	The corresponding results are reported in Table \ref{tab31inexact},  showing the efficiency of the inexact implementation of the proposed method and its superiority to the
	CGW algorithm.

In Tables \ref{tabbound1} and \ref{tabbound2} we report the bounds obtained in Theorem \ref{th3.3} for additional insight. In these two tables, we also used the symbol ``$\checkmark$" (``\xmark ") when
the condition  \eqref{eq28n} in Remark \ref{th3.2} is (not) satisfied. The results in Table \ref{tab31} illustrate that MRHSS needs to work with smaller $\gamma$ to be comparable with TSTMR even in the cases that sufficient condition \eqref{eq11}
holds, which usually corresponds to the case when $\gamma$ is quite close to $\mu^2$, see Table \ref{tabbound2}.

	\subsection{Performance of TSTMR as a regularization method}
	It is known that, in practice, the choice of the regularization parameter via GCV  can be expensive. An alternative is \fb{to exploit the regularization
	property of iterative methods, see for example \cite{Hansen2010}. Numerical experiments show that the TSTMR method acts as 
	an iterative regularization method. Therefore, in the sequel, we apply the proposed method for solving the following (non-regularized) block system}
	\begin{equation*}
		\underbrace {	\left[ {\begin{array}{*{20}{c}}
					I&A\\
					{ - {A^T}}&{0}
			\end{array}} \right]}_{K_0} \left[ {\begin{array}{*{20}{c}}
				e\\
				f
		\end{array}} \right]= \left[ {\begin{array}{*{20}{c}}
				g\\
				0
		\end{array}} \right].	
	\end{equation*}
	Evidently, solving the above system is mathematically equivalent to solving the normal equations $A^TAf=A^Tg$. To implement the TSTMR method,  we work with the splittings $K_0=\tilde{M}-\tilde{N}=\hat{M}-\hat{N}$ where we choose $\tilde{M}=I$ and $\hat M$ is defined by \eqref{mhat}	with $\gamma=0.001$.

	In the sequel, a test problem from \cite{Gazzola} is considered for which the matrix $A$ in \eqref{eq1}
	is non-square. We experimentally compare the performance of the TSTMR method with the CGLS method and a hybrid version of LSQR \cite{Chung}. To this end, the MATLAB codes in IR Tools package \cite{Gazzola} are exploited for solving the test problem.  
	In addition, we used the  \texttt{IRcgls} code in which the regularized solution is determined by terminating the CGLS iterations setting the regularization parameter equal to zero, i.e., $\mu=0$. 
	In this case the CGLS method is semi-convergent, see \cite{Hansen2010}. For a more comprehensive comparison, we further include the results obtained running the 
	\texttt{IRhybrid\_lsqr} code, which corresponds to a hybrid version of LSQR that applies a 2-norm penalty
	term to the projected problem. To be more specific, the regularization parameter was determined with two different strategies in  the hybrid version of LSQR, i.e., the discrepancy principle (DP) and  weighted GCV (WGCV). For clarification,  the terms \text{IRhybrid\_lsqr$^{*}$} and 	\text{IRhybrid\_lsqr$^{**}$} are respectively used to signify the cases that regularization parameter is determined by DP and WGCV. We refer the readers to \cite[Table 1]{Gazzola} for more details on the implementation of these approaches. 
	
	For all of the examined methods, the discrepancy principle is utilized as the stopping rule. More precisely, the iterations are terminated once
	\[
	\frac{{\|g -Ax_k\|}}{{\|g\|}}
	\leq \eta\cdot \texttt{NoiseLevel},
	\]
	here $\eta=1.01$ is a safety factor, and \texttt{NoiseLevel} stands for some estimate of the quantity ${\|e\|}/{\|\tilde{g}\|}$, where
	 $\tilde{g}$ denotes the (unknown) error--free vector associated with the right-hand side of \eqref{eq1}, i.e., $g=\tilde{g}+e$. The values for \texttt{NoiseLevel} are given in the captions of Tables \ref{tab9} and \ref{tab92}.

	\begin{example}\label{example4}
		{\rm
			We consider  the 2D fan-beam linear-detector tomography test problem from IR Tools toolbox. More precisely, we use the function $[A,\tilde{g},f]$ = {\bf fanlineartomo}$(n, \theta,p)$ inside the toolbox which exploits  the ``line model" to create a 2D X-ray tomography test	problem with an $n\times n$ pixel domain. The vector $\theta$ in the function includes  projection angles (in degrees) and the parameter $p$ is associated with the number of rays for each angle. Here we worked with default values of $\theta$ and $p$, i.e., $\theta = 0:2:358$ and $p = {\rm round(sqrt}(2)*n)$, resulting in a coefficient matrix $A$ of size $\ell \times n^2$ where $\ell={\rm length(\theta)}*p$.

	We report the numerical results for the above example in Tables \ref{tab9} and \ref{tab92}  for four different values of $n$. Although the proposed method  is not always the fastest in terms of CPU-time (especially for larger problems), its 
performance appears to be quite robust and the method results in smaller errors and
higher PSNRs than can be obtained with the other methods.
In the largest case, $A$ is approximately of size $25,000\times 10,000$ and contains just over $2,000,000$ nonzero entries.	
		
			\begin{table}[h]
	\centering
	\caption{Example \ref{example4}: Comparison results for  2D fan-beam linear-detector tomography test problem (\texttt{NoiseLevel}$=0.01$ and  and max$_{\rm itcg}=20$)}{
		\begin{tabular}{cllcccccc}
			\toprule   
			$n$&Determination of &Method&&&Err&CPU&Iter&PSNR\\
			&regularization parameter&&&&&&&\\
			\midrule
			\multirow{4}{*}{25} 
			&N/A&\text{TSTMR}                &&&  0.0320&0.01&2&42.3\\
			&N/A&\text{IRcgls}               &&&    0.0508&0.09&21&38.3\\
			&DP	&\text{IRhybrid\_lsqr$^{*}$}       &&&0.0490&0.10&23&38.6\\
			&WGCV	&\text{IRhybrid\_lsqr$^{**}$}   &&&  0.0532&0.11& 26&37.9\\
			\midrule
			\multirow{4}{*}{50} 	
			&N/A&\text{TSTMR}           &&& 0.0451&0.04&2&39.0\\
			&N/A	&\text{IRcgls}      &&& 0.0788&0.11&22&34.2\\
			&DP&\text{IRhybrid\_lsqr$^{*}$}   &&& 0.0759&0.12&24&34.6\\
			&WGCV&\text{IRhybrid\_lsqr$^{**}$} &&&0.0718&0.14&29&35.0\\
			\midrule
				\multirow{4}{*}{75} 	
			&N/A&\text{TSTMR}           &&&  0.0705&0.16&3&35.2\\
			&N/A	&\text{IRcgls}      &&& 0.1169&0.13&21&30.8\\
			&DP&\text{IRhybrid\_lsqr$^{*}$}   &&&  0.1148&0.14&23&30.9 \\
			&WGCV&\text{IRhybrid\_lsqr$^{**}$} &&& 0.1062&0.19&30&31.6\\
			\midrule
			\multirow{4}{*}{100} 
			&N/A	&\text{TSTMR}           &&&0.1268&0.29&3&30.3\\
			&N/A	&\text{IRcgls}          &&& 0.1749&0.17&21&27.5\\
			&DP	&\text{IRhybrid\_lsqr$^{*}$}  &&&  0.1706&0.19&23&27.7\\
			&WGCV&\text{IRhybrid\_lsqr$^{**}$}  &&& 0.1560&0.28&33&28.5\\
			\bottomrule
	\end{tabular}}%
	\label{tab9}%
\end{table}%

		\begin{table}[h]
	\centering
	\caption{Example \ref{example4}: Comparison results for  2D fan-beam linear-detector tomography test problem (\texttt{NoiseLevel}$=0.03$  and max$_{\rm itcg}=20$)}{
		\begin{tabular}{cllcccccc}
			\toprule   
			$n$&Determination of &Method&&&Err&CPU&Iter&PSNR\\
			&regularization parameter&&&&&&&\\
			\midrule
			\multirow{4}{*}{25} 
			&N/A&\text{TSTMR}                &&&0.0599&0.01&2&36.9\\
			&N/A&\text{IRcgls}               &&&0.1078&0.09&15&31.8\\
			&DP	&\text{IRhybrid\_lsqr$^{*}$}       &&&0.1155&0.08&15&31.2\\
			&WGCV	&\text{IRhybrid\_lsqr$^{**}$}   &&&0.1555&0.09&15&28.6\\
			\midrule
			\multirow{4}{*}{50} 	
			&N/A&\text{TSTMR}           &&&0.1066&0.04&2&31.6\\
			&N/A	&\text{IRcgls}      &&&0.1705&0.10&14&27.5\\
			&DP&\text{IRhybrid\_lsqr$^{*}$}   &&&0.1609&0.10&15&28.0\\
			&WGCV&\text{IRhybrid\_lsqr$^{**}$} &&&0.1901&0.12&17&26.6\\
			\midrule
					\multirow{4}{*}{75} 	
			&N/A&\text{TSTMR}           &&& 0.1675&0.10&2&27.7\\
			&N/A	&\text{IRcgls}      &&& 0.2231&0.11&12&25.2 \\
			&DP&\text{IRhybrid\_lsqr$^{*}$}   &&& 0.2254&0.12&12&25.1\\
			&WGCV&\text{IRhybrid\_lsqr$^{**}$} &&& 0.2122&0.15&15&25.6 \\
			\midrule
			\multirow{4}{*}{100} 
			&N/A	&\text{TSTMR}           &&&0.2406&0.21&2&24.7\\
			&N/A	&\text{IRcgls}          &&&0.2550&0.12&11&24.2\\
			&DP	&\text{IRhybrid\_lsqr$^{*}$}  &&& 0.2502&0.14&12&24.3\\
			&WGCV&\text{IRhybrid\_lsqr$^{**}$}  &&&0.2396&0.19&16&24.7\\
			\bottomrule
	\end{tabular}}%
	\label{tab92}%
\end{table}%
		
}
\end{example}

	\section{Conclusions}\label{sec5}
	In this paper we have introduced a new 
	class of two-step iterative methods for solving linear systems of equations. To construct these type of  methods, the approximate solution at each sub-step of a standard two-step iterative method is computed by minimizing  its residual norm over a certain two-dimensional subspace. The resulting approach is called the TSTMR method and, under certain conditions, can be proved to either converge in the limit or to break down after determining the exact solution in a finite number of steps. 
	Furthermore, we showed how the method can be adapted to solve a class of augmented systems corresponding to (shifted) normal equations 
	associated with least-squares problems arising from the discretization of ill-posed problems. An advantage of this method over established Krylov
	subspace solvers like GMRES is the low amount of memory required  for its implementation.
	
For well-posed problems, first, we demonstrated experimentally that the TSTMR method outperforms the MRHSS method for test problems of the convection-diffusion type examined in \cite{Yang}. We further used the proposed method for solving a block linear system of equations corresponding to a 3D coupled Stokes-Darcy flow problem. We observed that the proposed method provides more accurate solutions in comparison with the best approach in the recent paper \cite{BB2022} while both methods exhibit similar performance in terms of the required CPU times to satisfy a given stopping criterion. For discrete ill-posed problem arising in  tomography,
	the competitiveness of the TSTMR method with several other popular iterative schemes was also demonstrated by numerical experiments. 
	
	Future work should focus on analyzing the convergence properties of the proposed methods when implemented inexactly. While we found experimentally
	that the convergence of the method is not much affected by inexact inner solves, a better theoretical understanding of the convergence of inexact variants
	of TSTMR would be desirable. \fb{Another topic which may be worth investigating is the behavior of the proposed method when the linear system of equations 
	is singular but consistent.} \\
	
\section*{Acknowledgments}

\noindent The authors would like to thank Scott Ladenheim for providing the test problem in Example \ref{well2}. The work of M. Benzi was supported in part by ``Fondi per la Ricerca di Base" of the Scuola Normale Superiore di Pisa. Thanks also to three anonymous referees for their helpful suggestions.

	\bibliographystyle{abbrvnat}	

	\bibliography{refs.bib}

\end{document}